%% file: zonotopal-magic.tex
\documentclass[11pt,a4paper,reqno]{amsart}
\usepackage[margin=1.25in]{geometry}
\usepackage{amsmath, amsxtra, amsthm, amsfonts, amssymb, mathtools}
\usepackage{mathrsfs}
\usepackage{graphicx}
\usepackage{float}
\usepackage{url}
\usepackage[dvipsnames]{xcolor}
\usepackage{bbm}
\usepackage{bm}
\usepackage{tikz-cd}
\usepackage{tikz}
\usetikzlibrary{backgrounds}
\usepackage[cal=boondoxo,scr=euler]{mathalfa}
\usepackage{float}
\usetikzlibrary{decorations.pathreplacing, positioning}
\usetikzlibrary{decorations.markings}
\usetikzlibrary{shapes,positioning,intersections,quotes}
\usetikzlibrary{shapes,fit}
\usetikzlibrary{calc}
\usetikzlibrary{backgrounds, fit, calc, positioning, matrix, arrows.meta}
\usepackage{comment}
\usepackage[shortlabels]{enumitem} 
\usepackage{enumitem}
\usepackage{relsize}
\usepackage{stmaryrd}
\usepackage[all,cmtip]{xy}
\xyoption{arrow}
\usepackage[T1]{fontenc}
\usepackage{enumitem}
\usepackage{ytableau}

\usepackage[backref=page,linktocpage,hypertexnames=false]{hyperref}
\hypersetup{
    colorlinks,
    allcolors=teal
}

\usepackage[capitalize]{cleveref}

\newtheoremstyle{results}
{8pt}
{8pt}
{}
{}
{\bfseries}
{}
{.5em}
{}
\theoremstyle{results}
\newtheorem{thm}{Theorem}[section]
\newtheorem{proposition}[thm]{Proposition}
\newtheorem{corollary}[thm]{Corollary}
\newtheorem{lemma}[thm]{Lemma}

\newtheoremstyle{definitions}
{7pt}
{7pt}
{}
{}
{\bfseries}
{}
{.5em}
{}

\theoremstyle{definitions}
\newtheorem{definition}[thm]{Definition}

\newtheorem{example}[thm]{Example}

\newtheorem{remark}[thm]{Remark}

\crefname{thm}{theorem}{theorems}
\Crefname{thm}{Theorem}{Theorems}
\crefname{lemma}{lemma}{lemmas}
\Crefname{lemma}{Lemma}{Lemmas}
\crefname{proposition}{proposition}{propositions}
\Crefname{proposition}{Proposition}{Propositions}
\crefname{corollary}{corollary}{corollaries}
\Crefname{corollary}{Corollary}{Corollaries}
\crefname{definition}{definition}{definitions}
\Crefname{definition}{Definition}{Definitions}
\crefname{conjecture}{conjecture}{conjectures}
\Crefname{conjecture}{Conjecture}{Conjectures}
\crefname{example}{example}{examples}
\Crefname{example}{Example}{Examples}
\crefname{remark}{remark}{remarks}
\Crefname{remark}{Remark}{Remarks}

\def\Z{\mathbb{Z}}
\def\Q{\mathbb{Q}}
\def\R{\mathbb{R}}
\def\C{\mathbb{C}}

\def\P{\mathbb{P}}

\def\Z{\mathbb{Z}}

\def\calA{\mathcal{A}}
\def\calB{\mathcal{B}}

\def\calF{\mathcal{F}}

\def\calI{\mathcal{I}}

\def\calL{\mathcal{L}}

\def\calO{\mathcal{O}}

\newcommand{\abs}[1]{\left\lvert#1\right\rvert}

\newcommand{\B}{\textrm{B}}

\newcommand{\cl}{\textrm{cl}}

\newcommand{\G}{\textrm{G}}

\renewcommand{\Im}{\textrm{Im}}

\newcommand{\Pic}{\operatorname{Pic}}

\newcommand{\rk}{\operatorname{rk}}

\newcommand{\M}{\mathrm{M}}
\renewcommand{\emptyset}{\varnothing}

\renewcommand{\L}{\mathcal{L}}

\newcommand{\des}{\mathrm{des}}

\renewcommand{\G}{\mathcal{G}}

\usepackage{microtype}   

\newtheorem{theorem}[thm]{Theorem}
\crefname{theorem}{theorem}{theorems}
\Crefname{theorem}{Theorem}{Theorems}

\DeclareMathOperator{\supp}{supp}

\newcommand{\hstar}{h^{*}}
\newcommand{\cL}{\mathcal{L}}
\newcommand{\cO}{\mathcal{O}}
\newcommand{\Lines}{\mathcal{L}_2}
\newcommand{\DT}{\mathrm{T}}      
\newcommand{\capa}{\kappa}   
\newcommand{\dem}{\delta}    
\newcommand{\U}{\mathrm{U}}

\newcommand{\DHR}{\mathcal{D}}
\newcommand{\Cer}{\mathcal{C}}
\newcommand{\Mbar}{\overline{\mathcal M}}
\newcommand{\Zg}{\mathrm{Z}}
\newcommand{\W}{\mathcal{W}}      

\definecolor{gA}{named}{TealBlue}
\definecolor{gB}{named}{Violet}
\definecolor{gC}{named}{BurntOrange}
\colorlet{tAC}{gA!50!gC}
\colorlet{tAB}{gA!50!gB}
\colorlet{tBC}{gB!50!gC}

\newif\ifdraft \drafttrue

\usepackage{multirow,booktabs,array,graphicx}
\newlength{\blockht}

\input{parallelohedra}

\title{Magic positivity of Snapper polynomials \\ for matroids}
\author{Shiyue Li}
\address{Department of Mathematics, University of Michigan, Ann Arbor, Michigan 48109, USA}
\email{shiyueli\@umich.edu}

\begin{document}
\begin{abstract}
  In 1959, Snapper showed that the Euler characteristic of the tensor powers of a line bundle on a normal projective scheme is a polynomial, later named the \emph{Snapper polynomial}. 
  Positivity of coefficients of Snapper polynomials implies various notions of positivity of line bundles, which we study through the lens of magic positivity and real-rootedness. 
  We introduce zonotopal classes in the Grothendieck $K$-ring of vector bundles of the toric variety for any loopless matroid, and
  prove that their Snapper polynomials are magic positive. Our proof realizes such a Snapper polynomial as a weighted independence polynomial of the Dilworth truncation along certain lines of the matroid. As a consequence,
  their coefficients are positive, and their $h^{\ast}$-polynomials are real-rooted. 
  In the realizable case, this polynomial is the multigraded Hilbert polynomial of the wonderful variety embedded in a product of
  projective lines.

  We introduce analogous line bundles on the Deligne--Mumford--Knudsen moduli space $\overline{\mathcal M}_{0,n}$ and prove that their Snapper polynomials are magic positive. 
  For cotangent line bundles whose first Chern classes are distinct $\psi$-classes, which are not zonotopal, we nonetheless prove that their $h^{\ast}$-polynomials are real-rooted, whereas their Snapper polynomials are magic positive if and only if $n\leqslant7$.
  More generally, we introduce saturated and weakly saturated $K$-classes of matroids, which furnish a sufficient and a necessary condition
  for magic positivity of Snapper polynomials in terms of their dragon Hall--Rado polymatroids.
\end{abstract}
\maketitle
\raggedbottom

\setcounter{tocdepth}{1}
\tableofcontents

\section{Introduction}
\label{sec:intro}
One of the earliest results on the numerical invariants of line bundles is the
following theorem of Snapper in 1959 \cite{snapper1959multiples}:
\begin{quote}\itshape
  If $X$ is a normal projective scheme of dimension $d$ over an arbitrary field,
  together with any line bundles $\cL_1, \dots, \cL_k$, then the Euler
  characteristic
  $\chi\big(X, \cL_1^{\otimes x_1} \otimes \cdots \otimes \cL_k^{\otimes x_k}\big)$
  is a polynomial in $\Q[x_1, \dots, x_k]$ of total degree at most $d$.
\end{quote}
Soon after in 1960, Snapper defined the intersection number of line bundles using this polynomial, 
\cite[\S 4]{snapper1960polynomials}, later named \emph{Snapper polynomial}. It recovers Hilbert polynomials, virtual arithmetic
genera, intersection numbers, and volume polynomials. 

Positivity of line bundles can be viewed as positivity of the coefficients of a Snapper polynomial.  For instance, the numerical
characterization of ampleness, discovered by Nakai for smooth surfaces
\cite{nakai1963criterion} and developed for complete algebraic
varieties through the works of Grothendieck, Kleiman, Moi\v{s}ezon, Mumford, Nakai
and Zariski, can be phrased as follows \cite[Chapter~III]{kleiman1966ampleness}:
\begin{quote}\itshape
  A line bundle $\cL$ is ample if and only if the Snapper polynomial of $\cL|_Y$
  has positive leading coefficient for every closed integral subscheme
  $Y \subseteq X$.
\end{quote}

We investigate two positivity properties surrounding Snapper polynomials: magic positivity and
real-rootedness, which contain many $K$-theoretic positivity properties. Real-rootedness of a polynomial with nonnegative coefficients is widely regarded as the strongest positivity property, and it implies many other positivity properties such as log-concavity and unimodality. 
Yet another very strong positivity property is magic positivity. A polynomial $f \in R[q]$ is \emph{magic positive}, if
\[
  f(q)=\sum_{i=0}^d m_i q^i(1+q)^{d-i},\qquad m_i\geqslant0.
\]
By a theorem of Br\"and\'en \cite[Theorem 4.2]{branden2006linear}, and also by 
Beck, Jochemko and McCullough \cite[Corollary 4.5]{beck2019h}, 
if $f(q)$ is magic positive, then the Wagner transform $\W f(t)$ of $f$, defined by
\begin{equation}\label{eq:W}
 \sum_{q\geqslant0}f(q)t^q=\frac{(\W f)(t)}{(1-t)^{d+1}}, 
\end{equation} is a polynomial with only real zeros. 
As a consequence, the coefficients of $f$ are nonnegative, and those of $\W f$ are log-concave, unimodal. 
We refer to the beautiful handbook by \cite{branden2015unimodality} for further details. For a landscape of positivity properties of polynomials relevant to our paper, see \Cref{fig:positivity-hierarchy}.

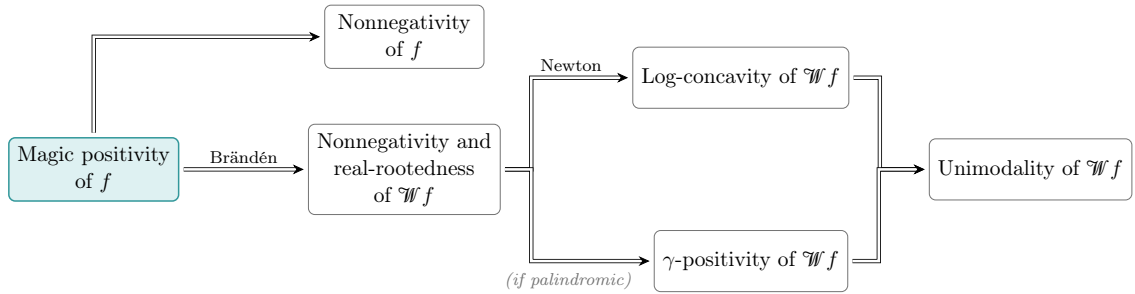
\begin{figure}[H]
  \centering
  \newcommand{\prov}[1]{{\scriptsize\itshape\textcolor{black!55}{#1}}}
  \scalebox{0.81}{%
  \begin{tikzpicture}[
      node distance = 0.9cm and 2.1cm,
      rectan/.style = {draw=black!55, rounded corners=3pt, align=center, inner sep=4pt,
                       font=\small, minimum height=10mm, minimum width=26mm},
      hot/.style    = {rectan, draw=gA!75!black, fill=gA!13, line width=0.7pt},
      lab/.style    = {font=\scriptsize, inner sep=1pt, anchor=south},
      arrow/.style  = {-{Stealth[length=5pt,width=4.5pt]}, double, double distance=1.2pt,
                       shorten >=2pt, shorten <=2pt}]

    \node (magic)  [hot] {Magic positivity\\of $f$};
    \node (rrn)    [rectan, right=of magic, minimum width=28mm]
                   {Nonnegativity and\\real-rootedness\\of $\W f$};
    \node (ehr)    [rectan, above=of rrn] {Nonnegativity\\of $f$};
    \node (lc)     [rectan, right=of rrn, yshift=1.5cm, xshift=0.05cm]  {Log-concavity of $\W f$};
    \node (gp)     [rectan, right=of rrn, yshift=-1.5cm, xshift=0.4cm] {$\gamma$-positivity of $\W f$};
    \node (uni)    [rectan, right=of lc, yshift=-1.5cm, xshift=-0.8cm] {Unimodality of $\W f$};

    \draw[arrow] (magic) |- (ehr);
    \draw[arrow] (magic) -- (rrn);
    \draw[arrow] (rrn.east) -- ++(0.5,0) |- (lc.west);
    \draw[arrow] (rrn.east) -- ++(0.5,0) |- (gp.west);
    \draw[arrow] (lc.east) -- ++(0.5,0) |- (uni.west);
    \draw[arrow] (gp.east) -- ++(0.5,0) |- (uni.west);

    \node[lab] at ($(ehr.west)+(-1.0,0.06)$) {};
    \node[lab] at ($(rrn.west)+(-1.05,0.06)$) {Br\"and\'en};
    \node[lab] at ($(lc.west) +(-1.0,0.06)$) {Newton};
    \node[lab] at ($(gp.west)+(-1.4,-0.5)$) {\prov{(if palindromic)}};
  \end{tikzpicture}}
  \caption{Positivity properties of a polynomial $f$ and $h(f; t)$}
  \label{fig:positivity-hierarchy} 
\end{figure}

We first focus on Snapper polynomials for wonderful varieties and matroids. Let $\M$ be a loopless matroid. Larson, Payne and
Proudfoot, and the author, defined its $K$-ring $K(\M)$ as the Grothendieck $K$-ring of vector bundles of the smooth and generally non-complete toric variety associated to the Bergman fan of $\M$, together with an Euler characteristic map
$\chi(\M,-)\colon K(\M)\to\Z$ \cite{larson2024k}. When $\M$ is realizable,
these agree with the Grothendieck $K$-ring of vector bundles, and the sheaf Euler characteristic of the smooth and projective wonderful variety of $\M$ introduced by \cite{deconcini1995wonderful}. Every lattice generalized permutohedron $P\subseteq\R^E$
defines a nef line bundle class $\cL_P\in K(\M)$. Its \emph{Snapper polynomial}
is
\[
  \chi_{\M,P}(q)=\chi\big(\M,\cL_P^{\otimes q}\big).
\]
If $\M$ is a Boolean matroid, $\chi_{\M,P}(q)$ is the Ehrhart polynomial of $P$.

Eur and Larson \cite{eur2023k} define the
$\hstar$-polynomial by
\begin{equation}\label{eq:hstar-def}
  \sum_{q\geqslant0}\chi_{\M,P}(q)t^q
  =\frac{\hstar_{\M,P}(t)}{(1-t)^{d+1}},
  \qquad
  \hstar_{\M,P}(t)=\sum_{i=0}^d\hstar_i t^i, 
\end{equation}
where $d=\deg\chi_{\M,P}$. Equivalently,
\begin{equation}
  \label{eq:Wf-change-of-basis}
  \chi_{\M,P}(q)=\sum_{i=0}^d\hstar_i\binom{q+d-i}{d}.
\end{equation}
If $\M$ is Boolean, then $\hstar_{\M,P}(t)$ becomes the $\hstar$-polynomial of $P$. For any matroid, 
Eur and Larson proved that the $\hstar$-vector for any simplicially nonnegative polytopes is the
Hilbert function of a standard graded Artinian algebra, and in particular, has only nonnegative coefficients. 
Recently, Eur, Fink, and Larson extended the result to any lattice generalized permutohedron. See \cite[Theorem 1.5]{eur2023k}, and \cite[Theorem B]{eurfinklarson2025}.

Our first aim is to introduce a family of \emph{zonotopal classes} of line bundles $\L_P$ in $K(\M)$, and to prove that the Snapper polynomials $\chi_{\M, P}(q)$ are magic positive. 
Every lattice generalized permutohedron $P \subseteq \R^{n}$ is uniquely a signed Minkowski sum of standard simplices with integer coefficients
\cite[Proposition 2.4]{ardila2010matroid}, up to translation: 
\[
  P=\sum_{\varnothing\ne S\subseteq E}y_S\Delta_S,
  \qquad \Delta_S=\operatorname{conv}\{e_i:i\in S\},
  \qquad y_S\in\Z. 
\]
The line bundle $\cL_{P}$ is nef, if the coefficients are nonnegative. 
For each flat $F$ of $\M$ rank at least $2$, we set 
\[
  \omega_{\M, P}(F) \coloneqq \sum_{\cl_{\M}(S)=F}y_S,
  \qquad \omega_{\M, P}(I) \coloneqq \prod_{F\in I}\omega_{\M, P}(F),
  \qquad \omega_{\M, P}(\varnothing) \coloneqq 1.
\]
For flats of rank $1$, the line bundles are trivial, so we omit them from the weights
and their support throughout. 

A line bundle $\L_{P}$ in $K(\M)$ is \emph{zonotopal}, if the weights $\omega_{\M, P}$ are nonnegative and supported on flats of rank $2$, or the \emph{lines} of $\M$. The corresponding $K$-class is a \emph{zonotopal class}. Zonotopal classes include $\cL_{P}$ for every graphical zonotope
\[
  P=\sum_{i<j}y_{ij}\Delta_{ij}\subseteq\R^E,
  \qquad y_{ij}\in\Z_{\geqslant0}, \qquad \Delta_{ij}=\operatorname{conv}\{e_i,e_j\}, 1 \leqslant i < j \leqslant n. 
\]
They also include classes represented by polytopes that are not zonotopes
(\Cref{ex:zonotopal-pair}).

\begin{theorem}[\Cref{thm:magic-zonotopal}]
  \label{thm:main}
  Let $\M$ be a loopless matroid and $\cL_P$ a zonotopal class in $K(\M)$.
  Then $\chi_{\M,P}(q)$ is magic positive. In particular,
  $\hstar_{\M,P}(t)$ only has real zeros. 
\end{theorem}

The proof identifies the Snapper polynomial with a weighted independence
polynomial. Let $\DT(\M)$ be the Dilworth truncation along the lines of $\M$:
a set $I$ of lines $\M$ is independent in $\DT(\M)$, if for every nonempty $J \subseteq I$,
\[
  \rk_{\M}\Big(\bigcup_{F\in J}F\Big) \geqslant |J| + 1. 
\]
We write $I \in \M$ when $I$ is independent in $\M$.

\begin{theorem}\label{thm:main-EC-independence}
  If $\cL_{P}$ is a zonotopal class in $K(\M)$, then 
  \[
    \chi_{\M,P}(q)=\sum_{I\in\DT(\M)}\omega_{\M, P}(I)q^{|I|}.
  \]
\end{theorem}

The key to this proof is to use a Hirzebruch--Riemann-Roch-type formula in \cite{larson2024k} to compute Euler characteristic using 
an elegant intersection number formula relying on the dragon-Hall--Rado condition \cite{backman2024simplicial}. 
Any shelling of the independence complex of the Dilworth truncation along the lines then expresses $\chi_{\M, P}(q)$ in the magic basis as a weighted $h$-polynomial with nonnegative coefficients.

For Boolean matroids, the formula above recovers Stanley's forest formula
for graphical zonotopes \cite{stanley1991zonotope}, and the real-rootedness of $h^{\ast}$-polynomial
recovers Beck, Jochemko and McCullough's
result for lattice zonotopes \cite{beck2019h}.

The magic positivity of $\chi_{\M, P}$ and the real-rootedness of $\hstar_{\M, P}$ do not hold for arbitrary line bundles given by lattice generalized permutohedra. Eur and
Larson identify $\hstar_{\M,\Delta(n-1,n)}$ with the $h$-polynomial of the
reduced broken circuit complex \cite[Example 5.8]{eur2023k}. For
$\M=\U_{n-1,n}$, it is $1+t+\cdots+t^{n-2}$, which is not real-rooted for
$n \geqslant 4$.

We next introduce analogous line bundles on the Deligne--Mumford--Knudsen compactification $\Mbar_{0,n}$ of rational curves
with $n$ marked points. For every subsets $S \subseteq [n-1]$ of size $3$, let the cross-ratio map define a collection of line bundles
\[
  \textrm{cr}_S\colon\Mbar_{0,n}\longrightarrow
  \Mbar_{0,S\cup\{n\}}\cong\P^1,
  \qquad \cL_S\coloneqq \textrm{cr}_S^*\cO_{\P^1}(1),
\]
and, for nonnegative integers $y = \begin{pmatrix}y_S\end{pmatrix}$, we define 
$\cL_y \coloneqq \bigotimes_S\cL_S^{\otimes y_S}$. We let  
\[
  \DT_n \coloneqq \left\{I\subseteq\binom{[n-1]}3:
  \left|\bigcup_{S\in J}S \cup \{n\}\right| \geqslant \abs{J} + 3
  \text{ for every nonempty }J\subseteq I\right\}, 
\]
which is the Dilworth truncation of the braid matroid along its
connected lines.

\begin{theorem}[\Cref{thm:m0n-triples}]\label{thm:main-m0n}
  For $n\geqslant3$ and nonnegative integers $y = \begin{pmatrix}y_S \end{pmatrix}$ for $S \in \binom{[n-1]}{3}$,
  \[
    \chi\big(\Mbar_{0,n},\cL_y^{\otimes q}\big)
    =\sum_{I\in\DT_n}\left(\prod_{S\in I}y_S\right)q^{|I|}.
  \]
  This polynomial is magic positive, and its $\hstar$-polynomial is real-rooted.
\end{theorem}

Next, we consider the cotangent line bundles on $\Mbar_{0, n}$, whose first Chern classes are the $\psi$-classes. Let
$\cL_i$ be the cotangent line bundle at the $i$th marked point, with
$c_1(\cL_i)=\psi_i$.

\begin{theorem}[\Cref{thm:psi-two}]\label{thm:main-psi}
  Let $n\geqslant3$, $d=n-3$, and $i\ne j$ in $[n]$. Then
  \[
    \hstar\big(\Mbar_{0,n},\cL_i\otimes\cL_j;t\big)
    =\sum_{k=0}^d\binom dk^2t^k.
  \]
  This polynomial is real-rooted. The Snapper polynomial
  $\chi\big(\Mbar_{0,n},(\cL_i\otimes\cL_j)^{\otimes q}\big)$ is magic
  positive if and only if $n\leqslant7$.
\end{theorem}

The formula above identifies $\hstar\big(\Mbar_{0,n},\cL_i\otimes\cL_j;t\big)$ with the type-$B$ Narayana polynomial
\cite{reiner1997noncrossing,simion2003type}, which is well-known to only have real zeros. 
Its real-rootedness follows from many properties, such as its relation to Legendre polynomials \cite{szego1975orthogonal}, and the fact that it is a Hadamard product of real-rooted polynomials \cite{wagner1992total}.

Finally, we introduce saturated and weakly saturated $K$-classes, and give sufficient and necessary conditions for magic positivity
beyond the zonotopal case. Suppose $\omega_{\M, P}\geqslant0$, and let
$\mathcal B$ be the set of maximal dragon Hall--Rado tuples supported on
$\supp(\omega_{\M, P})$; see \eqref{eq:dHR}. For a flat $F$ in the support of $\omega_{\M, P}$, define its
\emph{capacity} and \emph{demand}
\[
  \capa_F \coloneqq \rk_{\M}(F)-1,\qquad \dem_F \coloneqq \min_{B \in\mathcal{B}(\M)} B_F, 
\]
\begin{definition}\label{def:classes}
  A class of line bundle $\cL_P$ in $K(\M)$ with $\omega_{\M, P}\geqslant0$ is \emph{saturated} if
  $\omega_{\M, P}(F)\geqslant\capa_F$ for every supported flat, and
  \emph{weakly saturated} if $\omega_{\M, P}(F)\geqslant\dem_F$ for every supported
  flat.
\end{definition}

\begin{theorem}[\Cref{thm:sufficient,thm:necessary}]
  For pairs with nonnegative weights,
  \[
    \text{zonotopal}\ \Longrightarrow\ \text{saturated}\
    \Longrightarrow\ \text{magic positive}\
    \Longrightarrow\ \text{weakly saturated}.
  \]
\end{theorem}
The converses of the implications are not true (\Cref{ex:strict-implications}). 
The sufficient condition for magic positivity follows from a decomposition of the dragon
Hall--Rado polymatroid into boxes. For Boolean matroids, it answers one
direction of a question of Ferroni and Higashitani on Minkowski sums of
faces of a simplex \cite{ferroni2024examples}.

\subsection*{Acknowledgements} 
The author is thankful to Basile Coron for numerous stimulating discussions on related problems and techniques, as well as to Matt Larson, Chris Eur, Luis Ferroni for useful comments and suggestions on an earlier draft. 
This work was pursued independently and before the appearance of Chris Eur's
problem list \cite{eur2026problems}, and it makes progress to Problem 14 therein. Claude Opus 4.8/5 assisted with generating SageMath code for computations, producing figures, final-stage proofreading and wordsmithing; all mathematical content originates from the author. The author was supported by the NSF Grant DMS-1926686 at the IAS, 
and is currently supported by the Donald J. Lewis fellowship at the University of Michigan.

\section{Preliminaries}
\label{sec:prelim}
This section lays out the background for the three main ingredients of the proof of \Cref{thm:main}: matroids and their Dilworth truncations, the shellability and face enumeration polynomials of
independence complexes of matroids, as well as the $K$-rings and Euler characteristics of matroids.

\subsection{Matroids: independence complex and Dilworth truncation}
\label{sec:prelim-matroids}
In this subsection, we recall the basics of matroids using the definition of independent sets, as well as submodular functions.
For background on matroid theory, we refer to the classical textbook by \cite{oxley2011matroid}.

\begin{definition}
\label{def:matroid-independent-sets}
A \emph{matroid} $\M$ is a pair $(E, \calI)$ consisting of a finite ground set $E$ and a collection $\calI$ of subsets of $E$, called \emph{independent}, satisfying the following properties:
\begin{enumerate}[leftmargin=20pt, font=\normalfont]
  \item \emph{Nonempty:} $\varnothing \in \calI$.
  \item \emph{Hereditary:} If $X \in \calI$, and $Y \subseteq X$, then $Y \in \calI$.
  \item \emph{Independence augmentation property:} If $X, Y \in \calI$ with $\abs{X} < \abs{Y}$, then there exists an element $e \in Y \setminus X$ such that $X \cup \{e\} \in \calI$.
\end{enumerate}
\end{definition}

Given a matroid $\M$ on a finite ground set $E$, there is a rank function $\rk_{\M} \colon 2^{E} \to \Z_{\geqslant 0}$ given by
\[
\rk_{\M} \colon S \mapsto \max \{\abs{I} \colon I \subseteq S,  I \in \calI\}.
\]
The rank of the matroid $\M$ is $r = \rk_{\M}(E)$. Two distinct nonloop elements $i, j \in E$ are \emph{parallel}, if
$\rk_{\M}(\{i, j\}) = 1$, and \emph{nonparallel}, if
$\rk_{\M}(\{i, j\}) = 2$.
Furthermore, a \emph{flat} of $\M$ is a subset $F \subseteq E$ that is maximal
among the subsets of $E$ of rank $\rk_{\M}(F)$; the flats form a lattice, denoted
by $\calL(\M)$. For any set $S \subseteq E$, the \emph{closure} $\cl_{\M}(S)$ of $S$ in $\M$ is the smallest flat $F$ containing $S$, and satisfies $\rk_{\M}(S) = \rk_{\M}(F)$.
The flats of rank $2$ are the \emph{lines} of $\M$, denoted by $\Lines(\M)$.
If $e \in E$ and $\{e\} \notin \calI$, or equivalently, $\rk_{\M}(\{e\}) = 0$, then $e$ is called a \emph{loop}.
Throughout, $\M$ is loopless and its ground set is nonempty.

The \emph{independence complex} $\calI(\M)$ of $\M$ is the simplicial complex on $E$ whose
faces are the independent sets of $\M$; the independent sets which are the facets are called the \emph{bases}.
Since all bases have size $\rk_{\M}(E)$, $\calI(\M)$ is pure of dimension $\rk_{\M}(E) - 1$. For simplicity, we write $I \in \M$ to mean that $I$ is
an independent set of $\M$.

\begin{theorem}[Edmonds {\cite{edmonds1970submodular}}; see also {\cite{oxley2011matroid}}]\label{thm:edmonds}
  Let $E$ be a finite set and $f \colon 2^E \setminus \{\emptyset\} \to \Z$ a
  nonnegative, submodular function that is \emph{weakly increasing}, meaning $f(S) \leqslant f(T)$
  whenever $\emptyset \neq S \subseteq T$. 
  Then the set
  \[
    \calI_f \coloneqq \big\{ I \subseteq E :\ |J| \leqslant f(J) \text{ for every nonempty } J \subseteq I \big\}
  \]
  is the collection of independent sets of a matroid on $E$. In this case, we call this matroid $\M_f$.
\end{theorem}

A matroid that arises in the fashion above, which plays a central role in the present work, is the Dilworth truncation of $\M$, introduced by Dilworth in \cite{dilworth1944dependence}.
Since the rank function of a matroid is weakly increasing and submodular, the following lemma is immediate.
\begin{lemma}\label{lem:dilworth-matroid}
  Let $\calF \subseteq \calL(\M) \setminus \{\varnothing\}$ be a collection of
  nonempty flats of $\M$.  Then the function
  \[
    f(J) \;=\; \rk_{\M}\Big(\bigcup_{F \in J} F\Big) - 1,
    \qquad \emptyset \neq J \subseteq \calF,
  \]
  is nonnegative, weakly increasing and submodular.
\end{lemma}

\begin{definition}[Dilworth truncation, {\cite{dilworth1944dependence}}]
  \label{def:dilworth}
  Let $\calF \subseteq \calL(\M) \setminus \{\varnothing\}$ and let $f$ be the
  function of \Cref{lem:dilworth-matroid}.  The \emph{Dilworth truncation} of
  $\M$ along $\calF$ is the matroid $\DT_{\calF}(\M) \coloneqq \M_f$ on the
  ground set $\calF$ given by \Cref{thm:edmonds}. Explicitly, its independent
  sets are
  \begin{equation}\label{eq:dilworth}
    \DT_{\calF}(\M) \;=\; \Big\{ I \subseteq \calF \;:\;
      \rk_{\M}\Big(\bigcup_{F \in J} F\Big) \;\geqslant\; |J| + 1
      \ \text{ for every nonempty } J \subseteq I \Big\}.
  \end{equation}
\end{definition}
It follows immediately that $\DT_{\calF'}(\M) = \DT_{\calF}(\M)|_{\calF'}$ for every
  $\calF' \subseteq \calF \subseteq \calL(\M) \setminus \{\varnothing\}$.  We write $\DT(\M) = \DT_{\Lines(\M)}(\M)$ for the Dilworth truncation along the lines. 

If $\M$ is realized by linear forms on a vector space $L$, its elements
are represented by points in $\P(L^{\vee})$. The Dilworth truncation $\DT(\M)$ along the lines of $\M$ is the point configuration obtained by intersecting the lines of $\M$ by a generic hyperplane \cite{crapo1970foundations,mason1977matroids,kung1990dilworth}.

\begin{example}[Dilworth truncation of the Boolean matroids, {\cite[Theorem 3.2]{dilworth1944dependence}}]
  \label{ex:dilworth-boolean}
  Let $\M = \B_n$ be the Boolean matroid on $[n]$, whose flats are all
  subsets of $[n]$ and whose rank function is the cardinality of sets.  Its lines are the
  $2$-subsets of $[n]$, and so $\DT(\B_n)$ has ground set equal to the edge set of the
  complete graph $K_n$. For any set $I$ of edges, $I$ is independent in $\DT(\B_n)$ if and only if for every nonempty $J \subseteq I$,
  \[
  \rk_{\M}\Big(\bigcup_{F \in J} F \Big) = \Big| \bigcup_{F \in J } F \Big| \;\geqslant\; |J| + 1
  \]
  where the union on the left is a union of subsets in the original matroid $\M$ with ground set $[n]$ thought of as the vertex set of $K_n$.
  This requires that no subset of $I$ is a circuit. Thus $\DT(\B_n) = \M(K_n)$
  is the graphic matroid of $K_n$ of rank $n-1$, and the lattice of
  flats of $\DT(\B_n)$ is the partition lattice of $[n]$.
\end{example}

\subsection{Simplicial complexes: shellings and face enumeration polynomials}
\label{sec:prelim-shell}
Let $\Delta$ be a pure $(d-1)$-dimensional simplicial complex. For background on simplicial complexes, we refer to Stanley's textbook \cite{stanley1996combinatorics}. A
\emph{shelling} of $\Delta$ is an ordering $B_1, \dots, B_\ell$ of its facets such
that for each $k \geqslant 2$, the intersection of $B_k$ with the subcomplex
generated by $B_1, \dots, B_{k-1}$ is pure of dimension $d-2$; the complex is
\emph{shellable} if it admits a shelling.  A shelling of $\Delta$ determines the \emph{restriction
map} $\rho \colon \Delta(d) \to \Delta$
\begin{equation}
  \label{eq:restriction}
  \rho(B_k) \;\coloneqq\; \big\{\, v \in B_k \;:\; (B_k \setminus\{v\}) \subseteq B_j \text{ for some } j < k \,\big\},
\end{equation}
and
\begin{equation}
  \label{eq:shelling-partition}
  \Delta = \bigsqcup_{k = 1}^{\ell} [\,\rho(B_k), B_k\,], \text{ with } [\,\rho(B_k), B_k\, ] = \{F \mid \rho(B_k) \subseteq F \subseteq B_k\}
\end{equation}
Provan and Billera introduced vertex decomposability of a simplicial complex, which is a stronger property that implies shellability.
\begin{proposition}[\cite{provan1980decompositions}]\label{thm:vd-shellability}
  The independence complex of a matroid is vertex decomposable, and therefore, shellable.
\end{proposition}

For each integer $i\in \{0,1,\ldots,d\}$, we denote by $\Delta(i)$ the simplices of cardinality $i$ (that is, of dimension $i-1$), and we denote by $f_i(\Delta)$ the number of simplices in $\Delta(i)$.
The \emph{$f$-vector} of $\Delta$ is the $(d+1)$-tuple $(f_0(\Delta), f_1(\Delta), \dots, f_{d}(\Delta))$. The \emph{$h$-vector} of $\Delta$, denoted $(h_0(\Delta),\ldots,h_{d}(\Delta))$, is defined by the following equation
\[
  \sum_{i=0}^d f_{i}(\Delta)(y-1)^{d-i}
  \;=\;
  \sum_{i=0}^d h_i(\Delta) y^{d-i}.
\]
The $f$-polynomial $f(\Delta; x)$ and the $h$-polynomial $h(\Delta; y)$ are defined as
\[
  f(\Delta; x) \coloneqq \sum_{i=0}^d f_{i} \: x^i,
  \quad \text{ and } \quad
  h(\Delta; y) \coloneqq \sum_{i=0}^d h_i \: y^i.
\]
It is a classical result that shellability of a simplicial complex implies the nonnegativity of its $h$-vector. See \cite{stanley1996combinatorics}.
We are interested in \emph{vertex-weighted simplicial complexes}, and their \emph{weighted face enumeration polynomials}, in the following sense.
Let $w \colon \Delta(1) \to \Z_{\geqslant 1}$ be a weight function on the vertices of $\Delta$. It induces a weight function on all faces:
\[
w(I) \coloneqq \prod_{v \in I} w(v), \text{ with } w(\varnothing) = 1.
\] For each integer $i \in \{0, 1, \ldots, d\}$, the \emph{weighted $f$-vector} is the sum of weights over all simplices in $\Delta$ of size $i$:
\[
f_i(\Delta, w) = \sum_{I \in \Delta(i)} w(I).
\] Similarly, the \emph{weighted $h$-vector} is defined by the equation
\[
\sum_{i = 0}^{d} f_i(\Delta, w)(y-1)^{d-i} = \sum_{i = 0}^{d} h_i(\Delta, w) y^{d-i}.
\] When $w$ is the constant weight $1$, $f_i(\Delta, w) = f_i(\Delta)$, and $h_i(\Delta, w) = h_i(\Delta)$.
The weighted face enumeration polynomials are defined analogously,
\[
f(\Delta, w; x) \coloneqq \sum_{i = 0}^{d} f_i(\Delta, w) \: x^{i} = \sum_{I \in \Delta} w(I) \: x^{\abs{I}},
\qquad
h(\Delta, w; y) \coloneqq \sum_{i = 0}^{d} h_i(\Delta, w) \: y^{i}.
\]
Substituting $y$ with $1/y$ in the defining equation of the weighted $h$-vector gives
\begin{equation}
  \label{eq:weighted-h}
  h(\Delta, w; y)
  \;=\;
  \sum_{I \in \Delta} w(I) \: y^{\abs{I}} (1-y)^{d - \abs{I}}
  \;=\;
  (1-y)^{d} \: f\Big(\Delta, w; \frac{y}{1-y}\Big),
\end{equation}
and
\begin{equation}
  \label{eq:weighted-f}
  f(\Delta, w; x) = \sum_{i} h_i(\Delta, w) \; x^i (1 + x)^{d-i}.
\end{equation}
In other words, the weighted $h$-vector is the coefficient vector of the weighted $f$-polynomial in the magic basis.
A shelling gives the following weighted formula for the $h$-vector.
\begin{lemma}\label{lem:weighted-h}
  Let $\Delta$ be a pure, $(d-1)$-dimensional, shellable complex, and let
  $w \colon \Delta(1) \to \Z_{\geqslant 1}$ be a weight function on the vertices of $\Delta$.  Then the weighted $h$-polynomial
  \[
  h(\Delta, w; y) \;=\; \sum_{B \in \Delta(d)} w(\rho(B))\, y^{|\rho(B)|} \prod_{v \in B \setminus \rho(B)} \big(1 + (w(v) - 1) y\big),
  \]  In particular, the weighted
  $h$-vector $h(\Delta, w)$ is nonnegative.
\end{lemma}

\begin{proof}
  Fix a shelling with restriction map $\rho$. For a facet $B$ and
  $R=\rho(B)$, the contribution of the interval $[R,B]$ to
  \eqref{eq:weighted-h} is
  \[
    \sum_{R\subseteq I\subseteq B}w(I)y^{|I|}(1-y)^{d-|I|}
    =w(R)y^{|R|}\prod_{v\in B\setminus R}\big((1-y)+w(v)y\big).
  \]
  The identity follows by summing over the disjoint intervals in \eqref{eq:shelling-partition}.
  Since each factor has nonnegative coefficients with $w(v) \geqslant 1$, the coefficients of monomials in $y$ are nonnegative.
\end{proof}
When $w(v)=1$ for every vertex, the formula above recovers the classical fact that the $h$-vector of a
shellable complex is nonnegative, with $h_i$ counting the facets whose
restriction has size $i$.

\subsection{$K$-rings and Euler characteristics}
\label{sec:prelim-K}
In this subsection, we recall the basic facts about the $K$-ring and the Chow ring of a matroid.

Let $\M$ be a loopless matroid of rank $r$ on $E$, and let $X_{\M}$ be the smooth and generally incomplete toric variety associated to the Bergman fan of $\M$. The Chow ring of $\M$ is a graded algebra $A(\M)$, introduced by Feichtner and Yuzvinsky \cite{feichtner2004chow}, as the Chow ring of the toric variety
associated to the Bergman fan of $\M$. It has a degree map $\deg_{\M} \colon A^{r-1}(\M) \to \Z$ with respect to which the ring satisfies Poincar\'{e} duality.
In \cite{backman2024simplicial}, Backman, Eur, and Simpson introduced a simplicial presentation of the Chow ring, \footnote{The presentation here is a simplified version, but equivalent to the original simplicial presentation given by \cite{backman2024simplicial}; see \cite[Appendix]{larson2024k}.}
\[
A(\M) = \frac{\Z[h_F \mid \varnothing \ne F \in \calL(\M)]}{\langle h_A \mid \rk A=1 \rangle + \langle (h_F-h_{F\vee G})(h_G-h_{F\vee G}) \mid F,G \text{ incomparable}\rangle}.
\]
Similarly, the $K$-ring of $\M$ is defined as the Grothendieck $K$-ring of vector bundles (or equivalently, coherent sheaves) of $X_{\M}$.
It is generated by the classes $\cL_F$ of line bundles indexed by the
nonempty flats $F$ of $\M$ \cite[Lemma 1.20]{larson2024k}, and a ring map 
\[
\zeta_{\M} \colon K(\M) \to A(\M), \qquad 1 - [\calL_{F}]^{-1} \mapsto h_F = c_1(\calL_F)
\] gives an integral \emph{exceptional isomorphism} between $K(\M)$ and $A(\M)$ \cite[Theorem
1.16]{larson2024k}. Furthermore,
there is an \emph{Euler
characteristic} on $K(\M)$, defined by $\zeta_{\M}$ and the degree map
$\deg_{\M}$ by the Hirzebruch--Riemann--Roch-type formula
\begin{equation}\label{eq:HRR}
  \chi(\M, \xi) \;\coloneqq\; \deg_{\M}\!\big( \zeta_{\M}(\xi) \cdot \tau_{\M} \big),
  \qquad \tau_{\M} \;\coloneqq\; 1 + \alpha + \cdots + \alpha^{r-1},
\end{equation}
where $\alpha=h_E\in A^1(\M)$. 
Equivalently, we have the commutative triangle
\[
  \begin{tikzcd}[row sep=2em, column sep=1.4em]
    K(\M) \arrow[rr, "\zeta_{\M}"] \arrow[dr, "{\chi(\M,-)}"'] & & A(\M) \arrow[dl, "{\deg_{\M}(\,-\cdot\,\tau_{\M})}"] \\
    & \Z &
  \end{tikzcd}
\] When $\M$ is realized by a central and essential hyperplane arrangement $\calA$, the Euler characteristic $\chi(\M, \xi)$ coincides with the sheaf Euler characteristic of
the wonderful compactification $W_{\calA}$, introduced by De Concini and Procesi \cite{deconcini1995wonderful}. In that case, the exceptional isomorphism of \eqref{eq:HRR} is distinct from the classical Chern character, In the realizable case, $\tau_{\M}$ is distinct from the classical Todd class of the tangent bundle of $W_{\calA}$, and furnishes an alternative computation of Euler characteristics for wonderful varieties.

A \emph{lattice generalized permutohedron} $P \subseteq \R^{E}$ with $n = \abs{E} \geqslant 1$ is a lattice polytope whose only edge directions are $e_{i} - e_j$.
Its normal fan coarsens the $(n-1)$-dimensional permutohedral fan
$\Sigma_{n-1}$. The line bundles in $K(\M)$ are the Laurent monomials
$\bigotimes_{F} \cL_F^{\otimes a_F}$ in the generators above, indexed by integer
vectors $a = \begin{pmatrix}a_F\end{pmatrix}$ on the nonempty flats. Every lattice generalized permutohedron $P$ gives such a vector $a$ in the following way. Let the support function $\sigma \colon 2^{E} \to \Z$ of $-P$ be defined as
\[
  \sigma_P(S) \;\coloneqq\; -\min_{x \in P} \sum_{i \in S} x_i
\]
and let $\omega_{\M, P}$ be its M\"obius inversion along
the lattice of flats of $\M$: for every flat $F$ of rank at least $2$,
\begin{equation}\label{eq:moebius}
  \omega_{\M, P}(F) \;\coloneqq\; -\sum_{G \leqslant F} \mu_{\M}(G,F)\, \sigma_P(G)
\end{equation}
where $\mu_{\M}$ is the M\"obius function of $\calL(\M)$. We set
\begin{equation}\label{eq:LP-flats}
  \cL_P \;\coloneqq\; \bigotimes_{F} \cL_F^{\otimes \omega_{\M, P}(F)} \;\in\; K(\M),
  \qquad
  \chi_{\M,P}(q) \;\coloneqq\; \chi\big(\M, \cL_P^{\otimes q}\big) \in \Q[q].
\end{equation}
For a flat $A$ of rank $1$, the relation $h_A=0$ in the Chow ring $A(\M)$ implies that $\cL_A=1$ in $K(\M)$.
Therefore, only line bundles $\cL_F$ given by flats at least $2$ contribute to $\cL_P$ for any $P$. 
The polynomial is integer-valued, although its monomial coefficients need
not be integers.  Another description using a classical result of
Ardila, Benedetti and Doker \cite[Proposition 2.4]{ardila2010matroid} is as follows. Every
generalized permutohedron is uniquely a signed Minkowski sum of the standard
simplices $\Delta_S = \operatorname{conv}\{e_i \mid i \in S\}$,
\begin{equation}\label{eq:signed-minkowski}
  P \;=\; \sum_{S} y_S\, \Delta_S ,
  \qquad
  y_S = \sum_{T \subseteq S} (-1)^{\abs{S} - \abs{T}} \min_{x \in P} x(T) ,
\end{equation}
and every $y_S$ is an integer when $P$ is a lattice polytope.  Writing
$\cL_S \coloneqq \cL_{\cl_{\M}(S)}$ for a nonempty subset $S \subseteq E$, set
\begin{equation}\label{eq:LP}
  \cL_P \;=\; \bigotimes_{S} \cL_{S}^{\otimes y_S} \;\in\; K(\M) .
\end{equation}

The following lemma implies that \eqref{eq:LP-flats} and \eqref{eq:LP} define the same class.
\begin{lemma}
  \label{lem:support-general}
  Let $\M$ be a loopless matroid, and $P = \sum_{S} y_S\,\Delta_S$ with $y_S \in \Z$. For every flat $F$ of rank at least $2$,
  \[
    \omega_{\M, P}(F) \;=\; \sum_{\cl_{\M}(S) = F} y_S .
  \]
\end{lemma}

\begin{proof}
  Since $\sigma_{P}$ and $\omega_{\M, P}$ are additive functions on $P$, and so
  is the desired identity, it suffices to consider $P = \Delta_S$.  The minimum of
  $x(T)$ over $\Delta_S$ is attained at a vertex $e_i$ with $i \in S$, and
  $e_i(T) = 1$ precisely when $i \in T$; so that minimum is $1$ if
  $S \subseteq T$ and $0$ otherwise, whence $\sigma_{\Delta_S}(T) = -1$ if
  $S \subseteq T$ and $\sigma_{\Delta_S}(T) = 0$ otherwise.

  For a flat $G$ of $\M$, we have $S \subseteq G$ if and only if
  $\cl_{\M}(S) \leqslant G$ in the lattice of flats. Thus, in the M\"{o}bius
  transformation \eqref{eq:moebius}, only the flats above $\cl_{\M}(S)$
  contribute, and
  \[
    \omega_{\Delta_S}(F) \;=\;
    \sum_{\cl_{\M}(S) \,\leqslant\, G \,\leqslant\, F} \mu_{\M}(G,F).
  \]
  The defining identity for the M\"obius function makes this sum equal to
  $1$ when $F=\cl_{\M}(S)$ and $0$ otherwise. Additivity, including signed
  Minkowski sums, implies the result.
\end{proof}

We write $\omega(F) = \omega_{\M, P}(F)$, and for a set $I$ of flats of rank at least $2$, set 
\[
\supp(\omega) \coloneqq  \{F : \omega_{\M, P}(F) \neq 0\}, \qquad \omega(I) \coloneqq \omega_{\M, P}(I) = \prod_{F \in I} \omega_{\M, P}(F). 
\]

\begin{example}
When $\M$ is the Boolean matroid $\B_n$, every subset of $[n]$ is a flat and
$\mu(G,F) = (-1)^{\abs{F} - \abs{G}}$, the M\"obius transform in \eqref{eq:moebius} becomes
\[
  \omega_{\B_n, P}(S) \;=\; \sum_{G \subseteq S} (-1)^{\abs{S}-\abs{G}} \min_{x \in P} x(G)
\]
which gives the coefficients, for $|S|\geqslant2$, of the unique signed Minkowski decomposition $P = \sum_{S} y_S\, \Delta_S$
into standard simplices $\Delta_S = \operatorname{conv}\{e_i : i \in S\}$ by Ardila, Benedetti and Doker \cite[Proposition 2.4]{ardila2010matroid}.
For $q\geqslant0$, the higher cohomology of $\cL_P^{\otimes q}$ on the
permutohedral toric variety vanishes, and its global sections have a basis
indexed by $qP\cap\Z^n$. Consequently, $\chi_{\B_n,P}$ is the Ehrhart
polynomial of $P$ \cite[Section 9.4]{cox2011toric}.
\end{example}

Via the Hirzebruch--Riemann--Roch-type formula in \eqref{eq:HRR}, the Euler characteristic $\chi_{\M, P}(q)$ can be computed using intersection numbers of the generators $h_F$.
Backman, Eur and
Simpson \cite[Proposition 5.2.3]{backman2024simplicial}  give an elegant description of the intersection numbers,
relying on a condition that is a matroid generalization of Postnikov's dragon marriage
condition \cite[\S 5, \S 9]{postnikov2009permutohedra}. A tuple $k = (k_F)_F$ of
nonnegative integers indexed by the nonempty flats of $\M$ is \emph{dragon
Hall--Rado}, if for every nonempty $J \subseteq \supp(k)$,
\begin{equation}\label{eq:dHR}
  \rk_{\M}\Big(\bigcup_{F \in J} F\Big) \;\geqslant\; \Big(\sum_{F \in J} k_F\Big) + 1
  \tag{dragon Hall--Rado}.
\end{equation}
We write $\DHR(\M)$ for the set of all such tuples.  Taking $J = \{F\}$ in
\eqref{eq:dHR} gives the \emph{capacity} $k_F \leqslant \rk_{\M} F - 1$; in particular
$k_F = 0$ on flats of rank $1$.  Their intersection formula is the following.

\begin{theorem}[Backman--Eur--Simpson {\cite[Theorem 5.2.4]{backman2024simplicial}}]\label{thm:bes-intersection}
  Let $\M$ be a loopless matroid of rank $r$, let $F_1,\ldots,F_{r-1}$
  be nonempty flats, and put $k_F=|\{j:F_j=F\}|$. Then
  \[
    \deg_{\M}\big( h_{F_1} \cdots h_{F_{r-1}} \big)
    \;=\;
    \begin{cases}
      1 & \text{if } k \in \DHR(\M), \\
      0 & \text{otherwise.}
    \end{cases}
  \]
\end{theorem}

The second condition is the analogue given by \cite[Theorem 9.1]{larson2024k} for
$\Mbar_{0,n}$, or equivalently, the wonderful compactification of the braid arrangement complement with the minimal building set. A
tuple $c = (c_S)$ of nonnegative integers indexed by the subsets
$S \subseteq [n-1]$ with $|S| \geqslant 3$ is \emph{Cerberus}, if for every $0\ne c' \leqslant c$.
\begin{equation}\label{eq:cerberus-cond}
 \Big|\bigcup_{c'_S > 0} S \cup \{n\}\Big|  \;\geqslant\;  \sum_{S} c'_S + 3 
  \tag{Cerberus}
\end{equation}
We write $\Cer(n)$ for the set of Cerberus tuples.

Using the dragon Hall--Rado conditions, we recall the following formula for the Euler characteristic polynomial of the pair $\M, P$.
\begin{theorem}[Larson--Li--Payne--Proudfoot {\cite[Corollary 7.5]{larson2024k}}]\label{thm:dhr}
  Let $\M$ be a loopless matroid and $P$ a lattice generalized permutohedron on
  its ground set.  Then
  \begin{equation}
    \label{eq:dhr}
    \chi_{\M,P}(q) \;=\; \chi\big(\M, \cL_{P}^{\otimes q}\big) \;=\;
    \sum_{k \,\in\, \DHR(\M)} \ \prod_{F} \binom{\omega_{\M, P}(F)\, q + k_F - 1}{k_F}.
  \end{equation}
\end{theorem}

\subsection{Magic positivity and real-rootedness}
\label{sec:prelim-poly}
Given a polynomial $f(x) \in \R[x]$ of degree $d$, it is \emph{magic positive} if
\[
f(x) = \sum_{i} m_i x^i(1+x)^{d-i}, \text{ such that for every } i,  m_i \geqslant 0.
\] We call $\{x^i(1+x)^{d-i}\}_{i=0}^{d}$ the \emph{magic basis}, following \cite{ferroni2024examples}.
The coefficient vector in the expansion of $f$ in the magic basis is the \emph{magic vector}.
If $f(x) = \sum_{j=0}^{d} c_j x^j \in \R[x]$ is a polynomial of degree $d$, the \emph{magic transform} of $f(x)$ is the polynomial
\begin{equation}\label{eq:magic}
  m(f; y) \;\coloneqq\; (1-y)^{d}\, f\!\Big(\frac{y}{1-y}\Big) \;=\; \sum_{j=0}^{d} c_j\, y^{j} (1-y)^{d-j}.
\end{equation}

\begin{lemma}\label{lem:magic-transform}
  The coefficients of the magic transform $m(f; y)$ in the $y^i$ basis form the magic vector of $f$; that is,
  \[
    f(x) = \sum_{i=0}^{d} m_i\, x^i (1+x)^{d-i}, \qquad
    m(f; y) = \sum_{i=0}^{d} m_i\, y^i.
  \]
  In particular, $f$ is magic positive of degree $d$ if and only if $m(f; y)$ has
  nonnegative coefficients.  In that case
  \[
  m_0 = c_0 = f(0), \qquad m_1 = c_1 - d\,c_0\quad(d\geqslant1), \qquad m_d = (-1)^d f(-1) = h_d ,
  \]
  where $h_d$ is the top coefficient of the $\hstar$-vector of $f$.
\end{lemma}

\begin{proof}
  It suffices to prove that the linear transformation $m \colon f \mapsto m(f; y)$ sends $x^{i}(1+x)^{d-i}$ to $y^i$.
  Indeed, substituting $x = y/(1-y)$, so that $1 + x = (1-y)^{-1}$,
  \[
  m (x^i (1+x)^{d-i}; y) = (1-y)^{d}\, \frac{y^i}{(1-y)^i}\, \frac{1}{(1-y)^{d-i}} = y^i.
  \]
  The identity of the coefficients follows from the identities above.
\end{proof}

As a corollary, we see that the magic transform
of a weighted $f$-polynomial is the corresponding weighted $h$-polynomial, by the definition of the weighted $h$-vector in \eqref{eq:weighted-h}.
\begin{corollary}\label{cor:magic-f-to-h}
For any simplicial complex $\Delta$ with weight function $w \colon \Delta(1) \to \Z_{\geqslant 1}$,
\[
h(\Delta, w; y) = m(f(\Delta, w; x);y).
\]
\end{corollary}

For a polynomial $f\in\R[x]$ of degree $d$, let $\W f$ be determined by
\begin{equation}\label{eq:W}
 \sum_{q\geqslant0}f(q)t^q=\frac{(\W f)(t)}{(1-t)^{d+1}}.
\end{equation}
Thus $\W\chi_{\M,P}=\hstar_{\M,P}$. For a fixed $d$, the numerator on the
right is a linear function of $f$ among polynomials of degree at most $d$.

For $0\leqslant i\leqslant d$, define the refined Eulerian polynomial
\[
 A_{d,i}(t)=\sum_{\substack{\sigma\in S_{d+1}\\\sigma(1)=i+1}}
 t^{\des(\sigma)},
 \qquad
 \des(\sigma)=\big|\{j:\sigma(j)>\sigma(j+1)\}\big|.
\]
Their sum is the Eulerian polynomial $A_{d+1}(t)$, where
$A_d(t)=\sum_{\sigma\in S_d}t^{\des(\sigma)}$.

\begin{proposition}[{\cite[Theorem 4.2]{beck2019h}}]\label{prop:halfopen}
For $0\leqslant i\leqslant d$,
\[
 \W\big(x^i(1+x)^{d-i}\big)=A_{d,i}(t).
\]
\end{proposition}

The polynomial $x^i(1+x)^{d-i}$ counts lattice points in positive integral
dilations of the half-open cube $[0,1]^{d-i}\times[0,1)^i$.
Consequently, if $f(x)=\sum_{i=0}^d m_i x^i(1+x)^{d-i}$, then
\begin{equation}\label{eq:hstar-first-letter}
 (\W f)(t)=\sum_{i=0}^d m_i A_{d,i}(t).
\end{equation}
The refined Eulerian polynomials have nonnegative coefficients, and their
nonnegative linear combinations are real-rooted
\cite[Theorem 3.1 and Corollary 4.5]{beck2019h}. We therefore obtain the
following result of Br\"and\'en, in the form used here.

\begin{theorem}[{\cite[Corollary 4.5]{beck2019h}}]\label{thm:magic-rr}
If $f$ is magic positive, then $\W f$ has nonnegative coefficients and only
real, nonpositive roots.
\end{theorem}

Ferroni and Higashitani use this to prove Ehrhart positivity, magic positivity, and $\hstar$-real-rootedness for several classes of lattice polytopes
\cite[Theorem 4.19 and Remark 4.21]{ferroni2024examples} including zonotopes.

\section{Magic positivity of zonotopal classes}
\label{sec:proof}

\subsection{Zonotopes and zonotopal classes}
\label{sec:zonotopal-classes}
A \emph{zonotope} is a Minkowski sum of finitely many line segments,
\[
  Z(v_1, \ldots, v_\ell) \;=\; [0, v_1] + \cdots + [0, v_\ell] \;\subseteq\; \R^{n},
\]
equivalently, the projection of a cube; see
\cite[Lecture 7]{ziegler1995lectures}. The face lattice of $Z$ is determined by the oriented matroid of the vector configuration $v_1,\ldots,v_\ell$ \cite{bjorner1999oriented}.  A zonotope $Z$ can be tiled by half-open
parallelepipeds, one for each independent subset of $\{v_1, \ldots, v_\ell\}$
\cite[Theorem 54]{shephard1974combinatorial}; see, for instance,
\cite[Lemma 9.1]{beck2015computing} for a proof.  The tiling gives
Stanley's formula \cite[Theorem 2.2]{stanley1991zonotope}: the Ehrhart polynomial
of a lattice zonotope is $\sum_{X} d(X)\, q^{\abs{X}}$, summed over the linearly
independent subsets $X$ of the generators, where $d(X)$ is the greatest common
divisor of the maximal minors of the matrix whose rows are the elements of $X$.
From it, Beck, Jochemko and McCullough deduce
that the $\hstar$-polynomial of a lattice zonotope is real-rooted
\cite{beck2019h}.

Not every zonotope is a generalized permutohedron.  The normal fan of a zonotope $Z$
coarsens $\Sigma_{n-1}$ exactly when every generator $v_i$ is parallel to some
$e_i - e_j$, so the lattice zonotopal generalized permutohedra are precisely the
lattice translates of nonnegative integer combinations of standard line segments
\begin{equation}\label{eq:zonotope}
  P \;=\; \sum_{1 \leqslant i < j \leqslant n} y_{ij}\, \Delta_{ij},
  \qquad y_{ij} \in \Z_{\geqslant 0},
  \qquad \Delta_{ij} = \operatorname{conv}\{e_i, e_j\},
\end{equation}
that is, the polytopes of \eqref{eq:signed-minkowski} supported in
$|S| \leqslant 2$. Every such polytope is a lattice translate of a graphical zonotope.

\begin{example}[the permutohedron and the braid matroid]\label{ex:permutohedron}
  The $n$th permutohedron is the zonotope with all $y_{ij} = 1$
  \[
    \Pi_{n-1} \;=\; \sum_{1 \leqslant i < j \leqslant n} \Delta_{ij} \;\subseteq\; \R^{n}
  \]
  Its generators are the directions $e_i - e_j$ of the segments $\Delta_{ij}$, and a set of them is linearly
  independent exactly when the corresponding edges form a forest in $K_n$, which forms the braid matroid $\M(K_n)$.  Shephard's tiling
  therefore has one half-open parallelepiped per forest.  These are unimodular,
  and a forest with $k$ edges contributes $q^{k}$ lattice points to the $q$th dilation. Therefore, the Euler characteristic over $\B_n$ is the
  Ehrhart polynomial, which also coincides with the independence polynomial of the $n$th braid matroid
  \[
    \abs{q\,\Pi_{n-1} \cap \Z^{n}} \;=\; \chi_{\B_n,\, \Pi_{n-1}}(q)
    \;=\; \sum_{I \in \M(K_n)} q^{\abs{I}},
  \]
  For $n = 3$, see \Cref{fig:permutohedron}. The independent sets of $\M(K_3) = \U_{2,3}$ are $\{1, 2, 3, 12, 23, 13\}$, which correspond to the
  cells of the tiling, namely a point, three edges and three rhombi.  Therefore,
  $\chi_{\B_3, \Pi_2}(q) = 1 + 3q + 3q^2$, and $\chi_{\B_3, \Pi_2}(1) =  1 + 3 + 3 = 7$ which corresponds to the six vertices of the hexagon together with its center.
\end{example}
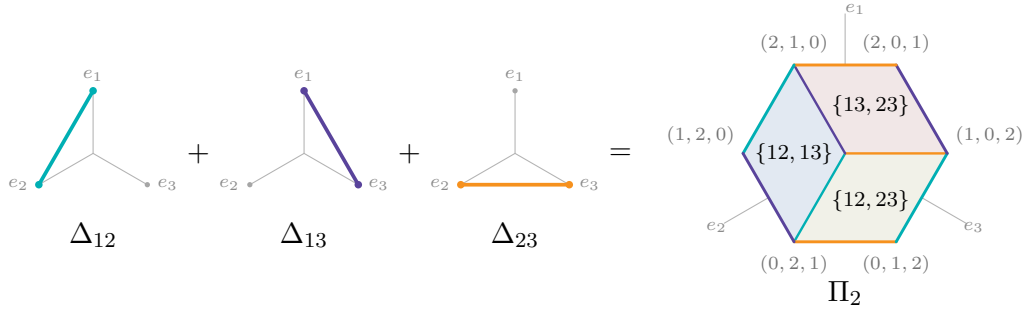
\begin{figure}[ht]
\centering
\begin{tikzpicture}[scale=0.90,
    seg/.style={line width=1.4pt, line cap=round},
    dot/.style={fill, circle, inner sep=1pt},
    bd/.style={line width=1.1pt, line join=round, line cap=round},
    inn/.style={line width=0.9pt, line cap=round},
    ax/.style={draw=black!30, line width=0.4pt, line cap=round},
    axdot/.style={fill=black!35, circle, inner sep=0.7pt},
    axlab/.style={black!45, font=\tiny}]
  \def\cy{1.299}
  \def\eA{(0,{\cy+0.92})}
  \def\eB{({-0.797},{\cy-0.46})}
  \def\eC{({ 0.797},{\cy-0.46})}
  \def\frame{%
    \draw[ax] (0,\cy) -- \eA;  \draw[ax] (0,\cy) -- \eB;  \draw[ax] (0,\cy) -- \eC;
    \node[axdot] at \eA {}; \node[axdot] at \eB {}; \node[axdot] at \eC {};
    \node[axlab, above=0pt] at \eA {$e_1$};
    \node[axlab, left=0pt]  at \eB {$e_2$};
    \node[axlab, right=0pt] at \eC {$e_3$};}
  \begin{scope}[shift={(0,0)}]
    \frame
    \draw[seg,gA] \eA -- \eB;
    \node[dot,gA] at \eA {}; \node[dot,gA] at \eB {};
    \node at (0,0.10) {$\Delta_{12}$};
  \end{scope}
  \node at (1.55,\cy) {$+$};
  \begin{scope}[shift={(3.1,0)}]
    \frame
    \draw[seg,gB] \eA -- \eC;
    \node[dot,gB] at \eA {}; \node[dot,gB] at \eC {};
    \node at (0,0.10) {$\Delta_{13}$};
  \end{scope}
  \node at (4.65,\cy) {$+$};
  \begin{scope}[shift={(6.2,0)}]
    \frame
    \draw[seg,gC] \eB -- \eC;
    \node[dot,gC] at \eB {}; \node[dot,gC] at \eC {};
    \node at (0,0.10) {$\Delta_{23}$};
  \end{scope}
  \node at (7.75,\cy) {$=$};
  \begin{scope}[shift={(10.3,0)}]
    \coordinate (V0) at (0,0);        \coordinate (V1) at (1.5,0);
    \coordinate (V2) at (2.25,1.299); \coordinate (V3) at (1.5,2.598);
    \coordinate (V4) at (0,2.598);    \coordinate (V5) at (-0.75,1.299);
    \coordinate (C)  at (0.75,1.299);
    \draw[ax] (C) -- ++(0,2.05);          \node[axlab] at (0.90,3.42) {$e_1$};
    \draw[ax] (C) -- ++(-1.775,-1.025);   \node[axlab] at (-1.14,0.22) {$e_2$};
    \draw[ax] (C) -- ++( 1.775,-1.025);   \node[axlab] at (2.64,0.22) {$e_3$};
    \fill[tAC!12] (V0) -- (V1) -- (V2) -- (C)  -- cycle;
    \fill[tAB!12] (V0) -- (C)  -- (V4) -- (V5) -- cycle;
    \fill[tBC!12] (C)  -- (V2) -- (V3) -- (V4) -- cycle;
    \draw[inn,gA] (V0) -- (C);   \draw[inn,gC] (C) -- (V2);  \draw[inn,gB] (C) -- (V4);
    \draw[bd,gC] (V0) -- (V1);   \draw[bd,gA] (V1) -- (V2);  \draw[bd,gB] (V2) -- (V3);
    \draw[bd,gC] (V3) -- (V4);   \draw[bd,gA] (V4) -- (V5);  \draw[bd,gB] (V5) -- (V0);
    \node[font=\tiny, black!55, below=1pt] at (V0) {$(0,2,1)$};
    \node[font=\tiny, black!55, below=1pt] at (V1) {$(0,1,2)$};
    \node[font=\tiny, black!55, above right=-1pt] at (V2) {$(1,0,2)$};
    \node[font=\tiny, black!55, above=1pt] at (V3) {$(2,0,1)$};
    \node[font=\tiny, black!55, above=1pt] at (V4) {$(2,1,0)$};
    \node[font=\tiny, black!55, above left=-1pt] at (V5) {$(1,2,0)$};
    \node at (1.125,0.60)  {\scriptsize $\{12,23\}$};
    \node at (0,1.299)     {\scriptsize $\{12,13\}$};
    \node at (1.125,1.999) {\scriptsize $\{13,23\}$};
    \node at (0.75,-0.75) {$\Pi_2$};
  \end{scope}
\end{tikzpicture}
\caption{The permutohedron $\Pi_2$.}
\label{fig:permutohedron}
\end{figure}

We introduce a general notion of matroids equipped with lattice generalized permutohedra, which yield a $K$-class in the $K$-rings of matroids.
This includes graphical zonotopes over arbitrary loopless matroids.

\begin{definition}[zonotopal class]\label{def:zonotopal-class}
  A line bundle $\L_{P}$ in $K(\M)$ is \emph{zonotopal}, if the weights $\omega_{\M, P}$ are nonnegative and supported on flats of rank $2$, or the \emph{lines} of $\M$. The corresponding $K$-class is a \emph{zonotopal class}.  
\end{definition}

Zonotopal classes include $\cL_{P}$ for every graphical zonotope
\[
  P=\sum_{i<j}y_{ij}\Delta_{ij}\subseteq\R^E,
  \qquad y_{ij}\in\Z_{\geqslant0}, \qquad \Delta_{ij}=\operatorname{conv}\{e_i,e_j\}, 1 \leqslant i < j \leqslant n. 
\] 

For a zonotopal pair, we write
\[
  \supp_2(\omega) \;=\; \supp(\omega) \cap \Lines(\M),
  \qquad
  \DT(\M,P) \;=\; \DT_{\supp_2(\omega)}(\M)
\]
for its set of \emph{supported lines} and for the Dilworth truncation of $\M$
along them.  Both depend on the pair, since $\omega = \omega_{\M,P}$ does.

A key combinatorial feature of a graphical zonotope $P$ is that $\omega_{\M, P}$ is supported on lines of $\M$, for any loopless matroid $\M$.
\begin{lemma}\label{lem:zonotope-is-zonotopal}
  For any loopless matroid $\M$, if a lattice generalized permutohedron $P$ in $\R^{E}$ is a zonotope, then $\cL_P$ is zonotopal in $K(\M)$.
\end{lemma}

\begin{proof}
  Let $P = \sum_{1 \leqslant i<j \leqslant n} y_{ij}\Delta_{ij}$ with $y_{ij}\in\Z_{\geqslant0}$. For every flat $F$ in $\M$ of rank at least $2$,
  \begin{equation}\label{eq:wF}
    \omega_{\M, P}(F) \;=\; \begin{cases}
    \sum\limits_{\cl_{\M}(i, j) = F} y_{ij} & \rk_{\M}(F) = 2, \\
    0 & \text{otherwise}.
    \end{cases}
  \end{equation}
  In particular, $(\M,P)$ is a zonotopal pair.
  By \Cref{lem:support-general},
  \[
  \omega_{\M, P}(F) = \sum_{\cl_{\M}(i, j) = F} y_{ij}.
  \]
  Since the closure of any $2$-subset has rank at most $2$, $\omega_{\M, P}(F) = 0$ for any flat $F$ of rank at least $3$.
\end{proof}

\begin{example}
  \label{ex:zonotopal-pair}
  The converse holds over the Boolean matroid, but fails in general. For example, let $\M = \U_{2,3} \oplus \B_1$
  on $\{1,2,3,4\}$ such that $F = \{1,2,3\}$ is a line in $\M$. Let
  $P = \Delta_{F}$ be the corresponding standard $2$-simplex on $[3]$. Then
  $\omega_{\M, P}(F) = 1$ and $\omega_{\M, P}(G) = 0$ for every other flat $G$ of rank at least two, so
  $\omega_{\M, P}$ is supported on the single line $F$ and $\cL_P$ is a zonotopal in $K(\M)$, even though $P$ is not a zonotope.
\end{example}

\begin{remark}
  Let us clarify the relationship between these zonotopal $K$-classes
and other related zonotopal objects in the recent literature. Holtz--Ron
\cite{holtz2011zonotopal} and Ardila--Postnikov \cite{ardila2010combinatorics}
introduced, for a vector configuration, a family of zonotopal algebras $Z_k$,
named after its zonotope, whose lattice point count, volume, and interior
lattice point count compute $\dim Z_k$ for $k=0$ (external), $k=-1$ (central),
and $k=-2$ (internal) respectively. Crowley--Larson \cite[Corollary~1.2]{crowley2026cohomological}
showed that the graded dual of $Z_k$ is the space of sections of
$\calO(1,\ldots,1)(k\alpha)$ on the augmented wonderful variety of the
configuration, where $\alpha$ is pulled back from the hyperplane class on the
projective completion. In the external case $k=0$, the class $\calO(1,\ldots,1)$, is the
augmented analogue of the classes, which can be considered as a special case of the classes that we consider here. For $k\neq 0$, $\calO(1, \ldots, 1)(k\alpha)$ produces classes
outside our family, as do the hierarchical zonotopal algebras of
Holtz--Ron--Xu \cite{holtz2012hierarchical} and Lenz \cite{lenz2012hierarchical}.
\end{remark}

\subsection{The Euler characteristic of a zonotopal class}
\label{sec:proof-formula}
In this subsection, we give a combinatorial formula for the Euler characteristics of tensor powers of zonotopal classes as a weighted independence polynomial of the Dilworth truncation of the matroid.

\begin{theorem}[\Cref{thm:main-EC-independence}]
  \label{thm:ec-zonotopal}
  Let $(\M, P)$ be a zonotopal pair.  Then
  \[
  \chi_{\M,P}(q) \;=\; \sum_{I \in \DT(\M)} \omega_{\M,P}(I)\, q^{\abs{I}} .
  \]
\end{theorem}

\begin{proof}
By \Cref{thm:dhr},
\[
 \chi_{\M,P}(q)=\sum_{k\in\DHR(\M)}
 \prod_F\binom{q\omega_{\M,P}(F)+k_F-1}{k_F}.
\]
If $k_F>0$ and $\omega_{\M,P}(F)=0$, the corresponding factor vanishes.
Every nonzero summand is therefore supported on $\supp_2(\omega)$.
For a line $F$, the dragon Hall--Rado condition gives
$k_F\leqslant\rk_{\M}F-1=1$. Hence every contributing tuple is the
indicator vector of a subset $I\subseteq\supp_2(\omega)$.
For such a tuple, the dragon Hall--Rado condition implies that, for every nonempty subset $J\subseteq I$,
\[
 \rk_{\M}\Big(\bigcup_{F\in J}F\Big)\geqslant |J|+1,
\]
which is precisely the independence condition in $\DT(\M,P)$.
The corresponding summand is
\[
\prod_{F\in I}q\omega_{\M,P}(F)=\omega_{\M,P}(I)q^{|I|}.
\]
The result follows by summing the formula over $\DT(\M,P)$ and noting that extending to $\DT(\M)$ only adds zero summands. 
\end{proof}

\subsection{Magic positivity of the Euler characteristic}
\label{sec:proof-magic}
\begin{theorem}[\Cref{thm:main}]\label{thm:magic-zonotopal}
  If $(\M, P)$ is zonotopal, then $\chi_{\M,P}(q)$ is magic positive, and
  its magic vector is the $\omega_{\M, P}$-weighted $h$-vector of the independence complex of $\DT(\M,P)$.  In
  particular, $\hstar_{\M,P}(t)$ is real-rooted.
\end{theorem}

\begin{proof}
  Let $\Delta$ denote the independence complex of $\DT(\M,P)$, and set
  $d = \rk\big(\DT(\M,P)\big)$.  By
  \Cref{thm:ec-zonotopal}, $d = \deg \chi_{\M,P}(q)$ and 
  \[
    \chi_{\M,P}(q) \;=\; \sum_{I \in \Delta} \omega(I)\, q^{\abs{I}},
  \]
  which is the weighted $f$-polynomial $f(\Delta, \omega_{\M,P}; q)$.
  Its magic transform \eqref{eq:magic} is
  \[
    m(\chi_{\M,P}; y) \;=\; (1-y)^d\, \chi_{\M,P}\Big(\frac{y}{1-y}\Big)
      \;=\; \sum_{I \in \Delta} \omega(I)\, y^{\abs{I}} (1-y)^{d - \abs{I}}.
  \]
  By \Cref{cor:magic-f-to-h}, the magic vector of $\chi_{\M,P}$ is the
  $\omega$-weighted $h$-vector of $\Delta$.
  Since $\Delta$ is the independence complex of a matroid, it is pure of dimension $d-1$ and shellable
  by \Cref{thm:vd-shellability}, and its vertex weights satisfy $\omega_{\M, P}(F) \geqslant 1$ by the
  definition of $\supp_2(\omega)$.  Therefore, \Cref{lem:weighted-h} implies that $m(\chi_{\M,P}; y)$
  has nonnegative integer coefficients, which by \Cref{lem:magic-transform} is magic
  positivity of $\chi_{\M,P}$.  Real-rootedness of $\hstar_{\M,P}$ then follows
  from \Cref{thm:magic-rr}.
\end{proof}

\begin{corollary}\label{cor:magic-extremes}
  In the situation of \Cref{thm:main}, if we denote by $\Delta$ the independence
  complex of $\DT(\M,P)$, which is pure of dimension $d - 1$ for
  $d = \rk\big(\DT(\M,P)\big)$, and fix a shelling of $\Delta$ with restriction map
  $\rho$, then
  \begin{enumerate}[leftmargin=20pt]
    \item \[m_0 = 1,\]
    \item
    \[
    m_1 \;=\; \sum_{F \in \supp_2(\omega)} \omega_F \;-\; d\qquad(d\geqslant1),
    \]
    \item
    \[
    m_d \;=\; \hstar_d \;=\; \sum_{B \in \Delta(d)} \omega(\rho(B)) \prod_{F \in B \setminus \rho(B)} (\omega_F - 1).
    \]
    If every $\omega_{\M, P}(F) = 1$, then $m_d$ is the number of bases with $\rho(B) = B$.
  \end{enumerate}
\end{corollary}

\begin{proof}
  By \Cref{lem:magic-transform}, $m_0 = c_0 = \chi_{\M,P}(0) = 1$ and
  $m_1 = c_1 - d\,c_0$, and by \Cref{thm:ec-zonotopal}
  $c_1 = \sum_{F \in \supp_2(\omega)} \omega_F$
  is the total weight of the vertices of $\Delta$. By
  \Cref{lem:magic-transform} again, $m_d = (-1)^d \chi_{\M,P}(-1) = \hstar_d$, and the
  coefficient of $y^d$ on the right-hand side of \Cref{lem:weighted-h} is the
  displayed sum, since the summand indexed by every $B$ has degree at most $d$ and coefficient of $y^d$ equal to 
  \[
  \omega(\rho(B))\prod_{F \in B \setminus \rho(B)}(\omega_{\M,P}(F) - 1).
  \]
\end{proof}

\subsection{Examples} We demonstrate the behavior of the Snapper polynomials of zonotopal classes across different matroids. 
\label{sec:examples}
\begin{example}\label{ex:small}
  Let $\M_1 = \B_3$ and
  $\M_2 = \U_{2,3}$. In $\B_3$, the lines are the three pairs of elements in $[3]$, each of
  weight $y_{ij}$. In $\U_{2,3}$, every pair has closure $[3]$, so $\omega$ is concentrated on the single line $[3]$. $\DT(\U_{2, 3},P) = \B_1$ has independence complex equal to the empty set together with a
  single vertex, which has higher and higher weight as $P$ varies from $\Delta_{12}$ to $2\Delta_{12} + \Delta_{13} + \Delta_{23}$ in the following table.
  {\footnotesize
  \[
    \setlength{\arraycolsep}{3pt}
    \renewcommand{\arraystretch}{1.15}
    \begin{array}{llllll}
      \hline
      P & \M & \DT(\M,P) & \chi_{\M,P}(q) & \hstar_{\M,P}(t) & m(\M,P) \\
      \hline
      \Delta_{12}
        & \B_3 & \B_1 & 1 + q & 1 & (1,0) \\
        & \U_{2,3} & \B_1 & 1 + q & 1 & (1,0) \\[2pt] \hline
      \Delta_{12} + \Delta_{23}
        & \B_3 & \B_2 & 1 + 2q + q^2 & 1 + t & (1,0,0) \\
        & \U_{2,3} & \B_1  & 1 + 2q & 1 + t & (1,1) \\[2pt] \hline
      \Delta_{12} + \Delta_{13} + \Delta_{23}
        & \B_3 & \U_{2,3} & 1 + 3q + 3q^2 & 1 + 4t + t^2 & (1,1,1) \\
        & \U_{2,3} & \B_1  & 1 + 3q & 1 + 2t & (1,2) \\[2pt] \hline
      2\Delta_{12} + \Delta_{13} + \Delta_{23}
        & \B_3 & \U_{2,3} & 1 + 4q + 5q^2 & 1 + 7t + 2t^2 & (1,2,2) \\
        & \U_{2,3} & \B_1 & 1 + 4q & 1 + 3t & (1,3) \\
      \hline
    \end{array}
  \]}
\end{example}

\begin{example}\label{ex:path}
   Let $P = \Delta_{12} + \Delta_{23} + \Delta_{34}$ be the path zonotope on $[4]$,
   and let $\M$ be uniform of rank $4$, $3$, $2$.
   \[
    \renewcommand{\arraystretch}{1.15}
    \begin{array}{lllll}
      \hline
      \M & \DT(\M,P) & \chi_{\M,P}(q) & \hstar_{\M,P}(t) & m(\M,P) \\
      \hline
      \B_4 & \B_3 & (1+q)^3 & 1 + 4t + t^2 & (1,0,0,0) \\
      \U_{3,4} & \U_{2,3} & 1 + 3q + 3q^2 & 1 + 4t + t^2 & (1,1,1) \\
      \U_{2,4} & \B_1 & 1 + 3q & 1 + 2t & (1,2) \\
      \hline
    \end{array}
  \]
  Note that the first two rows have the same $\hstar$-polynomial, even though their
  Euler characteristics $\chi_{\M, P}(q)$ have different degrees.
\end{example}

\newcommand{\zbox}[1]{\parbox[c][34pt][c]{36pt}{\centering\scalebox{0.50}{#1}}}
\newcommand{\gfour}[1]{\begin{tikzpicture}[scale=0.62,line width=0.8pt,
    baseline=(current bounding box.center)]
  \path (0,0) coordinate (a) (1,0) coordinate (b) (1,1) coordinate (c) (0,1) coordinate (d);
  \draw[gA!75!black] #1;
  \foreach \p in {a,b,c,d} \fill (\p) circle (2.4pt);
  \end{tikzpicture}}

\begin{example}[Fedorov's five parallelohedra]\label{ex:parallelohedra}
  Every zonotopal generalized permutohedron \eqref{eq:zonotope} arises from a graph
  \cite[\S 2]{postnikov2009permutohedra}.  Let $G$ be a graph
on the vertex set $[n]$, possibly with parallel edges, and let $y_{ij}$ be the
number of edges joining $i$ and $j$.  The \emph{graphical zonotope} of $G$ is
\[
  Z_G \;=\; \sum_{ij \in E(G)} \Delta_{ij}
        \;=\; \sum_{1 \leqslant i < j \leqslant n} y_{ij}\, \Delta_{ij}
        \;\subseteq\; \R^{n},
\]
a lattice zonotope of dimension $n - c(G)$, with $c(G)$ the number of components of $G$.

  Let $n = 4$ and $G$ be a connected graph on $4$ vertices. There are five types of these zonotopes, called Fedorov's
  five parallelohedra \cite{fedorov1885nachala}. \Cref{tab:parallelohedra} lists them with their Euler
  characteristics over the Boolean matroid $\B_4$, where $\chi_{\B_4,Z_G}$ is the
  Ehrhart polynomial of $Z_G$ and, by \Cref{thm:main}, counts the forests of $G$
  by number of edges.

\begin{table}[H]
  \centering
  \small
  \setlength{\tabcolsep}{5pt}
  \renewcommand{\arraystretch}{1.25}
  \begin{tabular}{llcccc}
    \hline
    & $\M$ & $\DT(\M,\Zg_G)$ & $\chi_{\M,\Zg_G}(q)$ & $\hstar(t)$ & $m(\chi_{\M, Z_G})$ \\
    \hline
    \multicolumn{6}{l}{\gfour{(a)--(b) (b)--(c) (c)--(d)} \hspace{2mm}
      \zbox{\zonoPath} \hspace{2mm} cube} \\
    & $\B_4$     & $\B_3$     & $1+3q+3q^2+q^3$ & $1+4t+t^2$ & $(1,0,0,0)$ \\
    & $\U_{3,4}$ & $\U_{2,3}$ & $1+3q+3q^2$     & $1+4t+t^2$ & $(1,1,1)$ \\
    \hline
    \multicolumn{6}{l}{\gfour{(a)--(b) (b)--(c) (a)--(c) (c)--(d)} \hspace{2mm}
      \zbox{\zonoPaw} \hspace{2mm} hexagonal prism} \\
    & $\B_4$     & $\M(G)$    & $1+4q+6q^2+3q^3$ & $1+10t+7t^2$ & $(1,1,1,0)$ \\
    & $\U_{3,4}$ & $\U_{2,4}$ & $1+4q+6q^2$      & $1+8t+3t^2$  & $(1,2,3)$ \\
    \hline
    \multicolumn{6}{l}{\gfour{(a)--(b) (b)--(c) (c)--(d) (d)--(a)} \hspace{2mm}
      \zbox{\zonoCycle} \hspace{2mm} rhombic dodecahedron} \\
    & $\B_4$     & $\M(G)$    & $1+4q+6q^2+4q^3$ & $1+11t+11t^2+t^3$ & $(1,1,1,1)$ \\
    & $\U_{3,4}$ & $\U_{2,4}$ & $1+4q+6q^2$      & $1+8t+3t^2$       & $(1,2,3)$ \\
    \hline
    \multicolumn{6}{l}{\gfour{(a)--(b) (a)--(c) (a)--(d) (b)--(c) (b)--(d)} \hspace{2mm}
      \zbox{\zonoDiamond} \hspace{2mm} elongated dodecahedron} \\
    & $\B_4$     & $\M(G)$    & $1+5q+10q^2+8q^3$ & $1+20t+25t^2+2t^3$ & $(1,2,3,2)$ \\
    & $\U_{3,4}$ & $\U_{2,5}$ & $1+5q+10q^2$      & $1+13t+6t^2$       & $(1,3,6)$ \\
    \hline
    \multicolumn{6}{l}{\gfour{(a)--(b) (b)--(c) (c)--(d) (d)--(a) (a)--(c) (b)--(d)} \hspace{2mm}
      \zbox{\zonoComplete} \hspace{2mm} truncated octahedron $\Pi_3$} \\
    & $\B_4$     & $\M(K_4)$  & $1+6q+15q^2+16q^3$ & $1+34t+55t^2+6t^3$ & $(1,3,6,6)$ \\
    & $\U_{3,4}$ & $\U_{2,6}$ & $1+6q+15q^2$       & $1+19t+10t^2$      & $(1,4,10)$ \\
    \hline
  \end{tabular}
  \caption{Fedorov's five parallelohedra as $K$-classes over $\B_4$ and over $\U_{3,4}$.}
  \label{tab:parallelohedra}
\end{table}
Over $\U_{3,4}$ every pair of the ground set $E$ forms a line, and no three of them are independent, so $\DT(\U_{3,4}) = \U_{2,6}$. 
For any of the five simple graphs in the table, by \Cref{thm:main}
\[
  \chi_{\U_{3,4},\, \Zg_G}(q) \;=\; 1 + \abs{E(G)}\, q + \binom{\abs{E(G)}}{2} q^2,
\]
where the forests of $G$ are replaced by the independent sets of $\DT(\U_{3, 4}) = \U_{2,6}$
supported on $E(G)$; therefore, the spanning trees and the coefficient of $q^3$ disappear, compared to the case over $\B_4$.
Since the remaining independent sets depend on $\abs{E(G)}$, the hexagonal prism and
the rhombic dodecahedron have the same Euler characteristics.
\end{example}

\begin{example}\label{ex:specializations}
As a corollary of the main theorem, $\hstar$-polynomials of zonotopal pairs are real-rooted.
Some classical Eulerian polynomials are recovered as $h^{\ast}(\M, P)$ for a zonotopal pair.
  \begin{enumerate}[leftmargin=20pt]
    \item For $d\geqslant1$ and $k\geqslant1$, if $\M = \B_{2d}$ and $P = k\big(\Delta_{12} + \cdots + \Delta_{2d-1,2d}\big)$,
      then $\hstar(\M,P)$ is the Eulerian polynomial of the
      $k$-colored permutations $\Z_k \wr S_d$
      \cite{steingrmsson1992permutation}, that is $A_d(t)$ for $k = 1$ and the
      type $B$ polynomial $B_d(t)$ for $k = 2$ \cite[Theorem 3.4]{brenti1994q}.
      It is also the $\mathbf{s}$-Eulerian polynomial for
      $\mathbf s = (k, 2k, \dots, dk)$ of \cite[Corollary 3.6]{savage2015s}.

    \item For $n\geqslant3$, if $\M = \B_n$ and $P$ is the cycle zonotope $\Delta_{12} + \cdots + \Delta_{n-1,n} + \Delta_{1n}$,
      then the Dilworth truncation $\DT(\M,P) \cong \U_{n-1,n}$ and
      $\hstar(\M,P) = A_n(t)$, the classical Eulerian polynomial of $S_n$.
  \end{enumerate}
\end{example}

\begin{remark}
Not every real-rooted polynomial in $\Z[t]$ arises as $h^{\ast}(\M, P)$ for a zonotopal pair.
\Cref{thm:magic-zonotopal} implies that the magic vector of a zonotopal pair is a
weighted $h$-vector, so its entries are nonnegative integers. Consider the
third Narayana polynomial $1 + 3t + t^2$, the $h$-polynomial of the graph
associahedron of the path graph on $3$ vertices, which is well known to be real-rooted.  For a zonotopal class $\cL_P$ in $K(\M)$ which has $\chi_{\M, P}$ of degree $d$, nonnegativity of the magic vector and
$m_0=1$ implies that
\[
  \hstar_1=\chi_{\M,P}(1)-(d+1)\geqslant2^d-d-1.
\]
Thus this polynomial could occur only for $d\leqslant2$, and its nonzero
quadratic coefficient forces $d=2$. In that degree, by
\eqref{eq:hstar-first-letter}, it would be $\sum_i m_i A_{2,i}$. Since the
$A_{2,i}$ form a basis, the expansion is unique, and solving the equation gives
$m = \big(1, \tfrac12, 1\big)$, which is not integral. 
\end{remark}

\section{Geometry of zonotopal classes}
\label{sec:geometry}
In this section, we explain the proof above using geometry of the wonderful varieties, introduced by De Concini and Procesi \cite{deconcini1995wonderful}.
Suppose $\M$ is realized by an essential central hyperplane arrangement
$\calA$ in a complex vector space $L$. There is a rational map
\[
\Psi_\calA \colon \P(L) \dashrightarrow \prod_{\emptyset \ne F \in \calL(\M)} \P(L^F), \qquad L^F = L/\bigcap_{e \in F} H_e.
\]
The wonderful compactification $W_{\calA}$ of the projective arrangement complement,
is defined as the closure of the image of the rational map
\begin{equation}\label{eq:wonderful}
W_{\calA} \coloneqq \overline{\Im(\Psi_{\calA})} \hookrightarrow \prod_{\emptyset \ne F \in \calL(\M)} \P(L^F),
\end{equation}
and the $K$-ring of $W_{\calA}$ is isomorphic to the $K$-ring of the matroid via the restriction map, $K(W_{\calA}) \cong K(\M)$ \cite[Proposition 1.13]{larson2024k}.
For a nonempty flat $F$ of $\M$, let
\[
  \pi_F \colon W_{\calA} \longrightarrow \P(L^F) \cong \P^{\,\rk_{\M}(F) - 1},
\]
be the projection of \eqref{eq:wonderful} to the factor indexed by $F$. The line bundles of the form $\cL_F$ defined as $\pi_F^{*}\,\cO(1)$ are the generators of $K(W_{\calA})$ of
\Cref{sec:prelim-K}.  For any nonempty $S \subseteq E$, since
\[
\bigcap_{e \in S} H_e = \bigcap_{e \in \cl_{\M}(S)} H_e,
\] one can define $L^S \coloneqq L^{\cl_{\M}(S)}$, and we define
$\cL_S \coloneqq \cL_{\cl_{\M}(S)}$.

Given a lattice generalized permutohedron
$P = \sum_{S} y_S \Delta_S$, the permutohedral fan $\Sigma_{n-1}$ refines the normal
fan of each standard simplex $\Delta_S$, so $\Delta_S$ defines a nef line bundle
\[
  \cL_{\Delta_S} \;\coloneqq\; \cO_{X_{n-1}}(\Delta_S)
\]
on the permutohedral variety $X_{n-1}$ \cite[Section 6.2]{cox2011toric}.  The same
construction applies to $P$ itself, and support functions add under Minkowski sum,
so \eqref{eq:signed-minkowski} gives
\begin{equation}\label{eq:OP}
  \cO_{X_{n-1}}(P) \;=\; \bigotimes_{\emptyset \ne S \subseteq E}
  \cL_{\Delta_S}^{\otimes y_S} .
\end{equation}
The Bergman
fan of $\M$ is a subfan of $\Sigma_{n-1}$, and the inclusion of fans induces an open
immersion $X_{\M} \hookrightarrow X_{n-1}$.
Composing it with the closed embedding
$W_{\calA} \subseteq X_{\M}$ \cite{katz2011realization} yields
\[
  \iota \colon W_{\calA} \hookrightarrow X_{n-1},
\]
whose restriction map on $K$-rings is a surjection $\iota^{*} \colon K(X_{n-1}) \to K(\M)$.

\begin{proposition}\label{prop:pullback}
  Let $\M$ be loopless and realized by an arrangement $\calA \subseteq L$.  For
  every nonempty $S \subseteq E$,
  \[
    \iota^{*} \cL_{\Delta_S} \;\cong\; \cL_{\cl_{\M}(S)} .
  \]
  As a consequence, for every lattice generalized permutohedron $P$,
  \[
    \iota^{*}\, \cO_{X_{n-1}}(P) \;\cong\; \cL_P .
  \]
\end{proposition}

\begin{proof}
  For the first statement, the polytope $\Delta_S$ is the polytope of the toric variety $\P^{|S|-1}$,
  and, since $\Sigma_{n-1}$ refines its normal fan, it defines a toric morphism
  \[
    p_{S} \colon X_{n-1} \longrightarrow \P^{|S|-1}
  \]
  such that $\cL_{\Delta_S} \cong p_S^{\ast}\cO(1)$.  Let $T^{n-1}$ denote the dense torus
  of $X_{n-1}$, which can be identified with the dense torus of $\P^{n-1}$.  On the dense torus $T^{n-1}$, the map $p_S$ is the coordinate projection
  \[
    [x_e]_{e \in E} \longmapsto [x_e]_{e \in S} .
  \]

  By the definition of the closure of $S$ in the matroid $\M$,
  $\bigcap_{e \in S} H_e = \bigcap_{e \in \cl_{\M}(S)} H_e$, and
  $L^{S} = L^{\cl_{\M}(S)}$.  The coordinates $x_e$ with $e \in S$ embed $L^{S}$
  into $\C^{S}$, realizing $\P(L^{S})$ as a linear subspace of $\P^{|S|-1}$.
  Therefore, the aforementioned maps form a commutative square
  \[
    \begin{tikzcd}[row sep=2.2em, column sep=3em]
      W_{\calA} \arrow[r, hook, "\iota"] \arrow[d, "\pi_{\cl_{\M}(S)}"'] & X_{n-1} \arrow[d, "p_S"] \\
      \P(L^{\cl_{\M}(S)}) \arrow[r, hook] & \P^{|S|-1}.
    \end{tikzcd}
  \]
  Pulling back $\cO(1)$ on $\P^{\abs{S}-1}$ along the two paths
  now gives the first statement, and since $\cO(1)$ on $\P^{\abs{S}-1}$ restricts to $\cO(1)$ on $\P(L^{\cl_{\M}(S)})$, we have
  \[
    \iota^{*}\cL_{\Delta_S} \;\cong\; (p_S \circ \iota)^{*}\,\cO(1)
    \;=\; \pi_{\cl_{\M}(S)}^{*}\,\cO(1) \;=\; \cL_{\cl_{\M}(S)} .
  \]

  For the second statement, expand $\cO_{X_{n-1}}(P)$ by \eqref{eq:OP}, apply the
  previous statement to each factor, and group the subsets by their closure:
  \[
    \iota^{*}\, \cO_{X_{n-1}}(P)
    \;\cong\;
    \bigotimes_{\emptyset \ne S \subseteq E} \cL_{\cl_{\M}(S)}^{\otimes y_S}
    \;=\;
    \bigotimes_{\emptyset \ne F \in \calL(\M)}
      \cL_{F}^{\otimes\, \sum_{\cl_{\M}(S) = F}\, y_S} .
  \]
  For flats of rank at least two, \Cref{lem:support-general} identifies the
  exponent with $\omega_{\M,P}(F)$; the remaining factors are trivial. Thus
  the product is $\cL_P$ by \eqref{eq:LP-flats}.
\end{proof}

We next describe the geometry behind the proof of \Cref{thm:ec-zonotopal}.
Let $(\M, P)$ be zonotopal. By definition, $\omega_{\M, P}$ is supported on $\Lines(\M)$, and we write
\[
  \omega \;=\; (\omega_F)_{F \in \Lines(\M)} \;\in\; \Z_{\geqslant 0}^{\Lines(\M)} .
\]  For a line $F$ in $\M$, the factor $\P(L^F)$ of \eqref{eq:wonderful} is a $\P^1$. Write
\[
  \pi_{\Lines(\M)} = (\pi_F)_{F \in \Lines(\M)} \colon W_{\calA} \longrightarrow
  \prod_{F \in \Lines(\M)} \P^1 ,
\]
and similarly, write
\[
\pi_2 = (\pi_S)_{\abs{S} = 2} \colon X_{n-1} \to \prod_{S \subseteq E, \abs{S} = 2}\P^1
\]
We have a commutative square
\[
  \begin{tikzcd}[row sep=2.2em, column sep=3em]
    W_{\calA} \arrow[r, hook, "\iota"] \arrow[d, "\pi_{\Lines(\M)}"'] & X_{n-1} \arrow[d, "\pi_2"] \\
    \prod\limits_{F \in \Lines(\M)} \P^1 \arrow[r, hook] & \prod\limits_{S \subseteq E, \abs{S} = 2} \P^1 .
  \end{tikzcd}
\]

\begin{corollary}\label{cor:embedding}
  The map $\pi_{\Lines(\M)}$ is a closed embedding, and
  \[
    \cL_P \;\cong\; \pi_{\Lines(\M)}^{*}\,\cO(\omega) .
  \]
\end{corollary}

\begin{proof}
  The product of the pair projections
  $\pi_2\colon X_{n-1}\to\prod_{|S|=2}\P^1$ is a closed embedding.
  Indeed, on the standard affine chart corresponding to an ordering
  $\sigma(1),\ldots,\sigma(n)$, its coordinate functions include the
  consecutive ratios $x_{\sigma(i)}/x_{\sigma(i+1)}$, which generate
  the coordinate ring of that chart. The corresponding affine charts
  of the target therefore give closed immersions, and $\pi_2$ is proper.

  For a nonparallel pair $S$, its projection on $W_{\calA}$ is
  $\pi_{\cl_{\M}(S)}$ followed by a projective linear isomorphism;
  for a parallel pair it is constant. Hence $\pi_2\circ\iota$ factors
  through $\pi_{\Lines(\M)}$ via the closed embedding in the displayed
  square. It follows that $\pi_{\Lines(\M)}$ is a closed embedding.
  The identity for $\cL_P$ follows from \Cref{prop:pullback}, since
  rank-one flat factors are trivial.
\end{proof}

For a closed subvariety $X\subseteq\P^{n_1}\times\cdots\times\P^{n_\ell}$,
its multigraded Hilbert polynomial is
\[
  H_X(k)=\chi\big(X,\cO(k)|_X\big).
\]
If $d=\dim X$, its top-degree part is
$\sum_{|a|=d}\deg_a(X)k^a/a!$; the integers $\deg_a(X)$ are its
multidegrees \cite[Definition 2.7]{castillo2020when}. For a product of
projective lines, only squarefree $a$ contribute, so these are the ordinary
coefficients of the top-degree part.

By \Cref{cor:embedding}, $\pi_{\Lines(\M)}$ embeds $W_{\calA}$ in
$\prod_{\Lines(\M)}\P^1$ as the closure of the image of the rational map
\[
  \P(L) \dashrightarrow \prod_{F \in \Lines(\M)} \P(L/L_F), \qquad
  L_F = \bigcap_{e \in F} H_e ,
\]
since $\pi_{\Lines(\M)}(W_{\calA})$ is closed and contains the image of the
complement of the arrangement.  Binglin Li computed the multidegree and the Hilbert
polynomial of the closure of the image of a projective space under linear
projections \cite[Theorem 1.1]{li2018images}.  In our situation, for $J \subseteq \Lines(\M)$, we define
\[
  Z_J \;\coloneqq\; \prod_{F \in J}\P^1 \times \prod_{F \notin J}\{\mathrm{pt}\},
\]
so that $\chi\big(Z_J, \cO(k)|_{Z_J}\big) = \prod_{F \in J}(k_F+1)$ and
$Z_I \cap Z_J = Z_{I \cap J}$.

\begin{samepage}
\begin{theorem}\label{thm:li}
  Let $\M$ be a loopless matroid of rank $r \geqslant 2$ realized by a central and essential arrangement
  $\calA \subseteq L$.  The multigraded Hilbert polynomial of
  \[
  W_{\calA} \subseteq \prod_{\Lines(\M)}\P^1
  \] is the independence
  polynomial, or equivalently, the $f$-polynomial of the independence complex, of the Dilworth truncation $\DT(\M)$,
  \begin{equation}\label{eq:hilb}
    H_{W_{\calA}}(k) \;=\; \sum_{I \in \DT(\M)}\ \prod_{F \in I} k_F.
  \end{equation}
  In particular, the multidegree of $W_{\calA}$ is the indicator function of the bases of
  $\DT(\M)$.
\end{theorem}
\end{samepage}

\begin{proof}
  Since $\rk(\M) \geqslant 2$ and $\M$ is loopless, every $e \in E$ has
  an element $f$ not parallel to it, and then $e$ is contained in
  $\cl_{\M}(\{e, f\})$, and so the union of lines of $\M$ contains $E$ and
  \[
  \bigcap_{F \in \Lines(\M)} L_F = \bigcap_{e \in E} H_e = \{0\}.
  \]
  Since $L$ has dimension $r$, for a nonempty set $I$ of lines
  \[
    \dim \bigcap_{F \in I} L_F \;=\; r - \rk_{\M}\Big(\bigcup_{F \in I} F \Big) .
  \]
  \cite[Theorem 1.1]{li2018images} states that, for a nonempty set $S$ of bases of $\DT(\M)$, each of which is a set of lines of $\M$,
  and if we denote $I_S = \bigcap_{B \in S} B$, then the multidegree of $W_{\calA}$ is the indicator
  function of the bases of $\DT(\M)$, and
  \[
    H_{W_{\calA}}(k) \;=\; \sum_{\emptyset \ne S \subseteq \mathcal B(\DT(\M))}
      (-1)^{|S|-1} \prod_{F \in I_S} (k_F + 1) ,
  \]
  Each term expands as
  \[
    \prod_{F \in I_S} (k_F + 1) \;=\; \sum_{I \subseteq I_S}\ \prod_{F \in I} k_F .
  \]
  A set $I$ of lines satisfies $I \subseteq I_S$ for some nonempty $S \subseteq \calB(\DT(\M))$ if and only
  if $I \in \DT(\M)$, and for such $I$
  \[
    \sum_{\substack{S \ne \emptyset \\ I \subseteq I_S}} (-1)^{|S|-1} \;=\; 1 ,
  \]
  by the M\"{o}bius function on the Boolean poset of subsets of the set of bases
  containing $I$.  The desired equality
  \eqref{eq:hilb} follows.
\end{proof}

Specializing $k=q\omega$ in \eqref{eq:hilb} recovers
\Cref{thm:ec-zonotopal} for realizable matroids.

\begin{remark}\label{rem:degeneration}
  The multidegrees of $W_{\calA}$ are zero or one. Brion's theorem
  \cite[Theorem 1]{brion2003multiplicity} therefore gives a flat degeneration
  to the reduced union $Y=\bigcup_B Z_B$, over the bases of $\DT(\M)$,
  after choosing the fixed points in the factors. This union is defined by
  the Stanley--Reisner ideal of the independence complex, in the corresponding
  multihomogeneous coordinates. It has the same multigraded Hilbert
  polynomial as $W_{\calA}$, giving a geometric interpretation of
  \eqref{eq:hilb}.
\end{remark}

\begin{remark}
In recent work, Beck,
Klivans and Ross \cite{beck2026matroidal} give a description of the Euler characteristics of matroids as the \emph{matroidal twist} of Brion's formula \cite{brion1988points}.  They prove a deletion--contraction formula
\cite[Theorem 1.2]{beck2026matroidal} for the Euler characteristics on $K(\M)$. It would be interesting to give an alternative proof of \Cref{thm:ec-zonotopal} for zonotopal classes using their formula.
\end{remark}

\section{Zonotopal classes on \texorpdfstring{$\overline{\mathcal M}_{0,n}$}{the moduli space of stable rational curves}}
\label{sec:m0n}

The Deligne--Mumford--Knudsen compactification $\Mbar_{0,n}$ of stable rational
curves with $n$ marked points is a wonderful variety of the braid arrangement with
the minimal building set.  Via the framework of wonderful varieties, Larson, Li,
Payne and Proudfoot studied Euler characteristics of $\Mbar_{0,n}$ and its Snapper
polynomials \cite[\S 9]{larson2024k}.  Following \cite{larson2024k} and our
previous analysis of zonotopal classes, we introduce analogous zonotopal classes
on $\Mbar_{0,n}$, and show that their Euler characteristics are magic positive.
We also show that the $\hstar$-polynomial of the product $\cL_i \otimes \cL_j$ of
two cotangent line bundles is real-rooted for every $n$, although
$\cL_i \otimes \cL_j$ is magic positive only for $n \leqslant 7$.

Fix $n \geqslant 3$.  For $i \in [n]$ let $\cL_i$ be the cotangent line bundle at
the $i$-th marked point, with $c_1(\cL_i) = \psi_i$.  For $S \subseteq [n-1]$ with $|S|\geqslant2$, the
forgetful map
\[
  f_{S \cup n} \colon \Mbar_{0,n} \longrightarrow \Mbar_{0,S \cup n}
\]
remembers only the points marked by $S \cup \{n\}$, and we set
$\cL_S \coloneqq f_{S \cup n}^{*}\cL_n$.  This line bundle is trivial when $|S| = 2$,
and by \cite[\S 1.4]{larson2024k} the classes $c_1(\cL_S) = f^*_{S \cup n}\psi_n$
with $|S| \geqslant 3$ form a basis of $\Pic(\Mbar_{0,n})$.  Therefore, every line bundle on
$\Mbar_{0,n}$ is 
\begin{equation}
  \label{eq:line-bundle-Mbar-0n}
  \cL_y \;=\; \bigotimes_{S \subseteq [n-1],\ |S| \geqslant 3} \cL_S^{\otimes y_S},
  \qquad y = (y_S) \in \Z^{\{S\,\mid\,|S| \geqslant 3\}} .
\end{equation}
For $S = \{a,b,c\}$ of size $3$ the set $S \cup \{n\}$ has four elements and
$\Mbar_{0,S \cup n} \cong \P^1$ by the cross-ratio of the four marked points, so
$f_{S \cup n}$ is the cross-ratio map, given on the smooth locus in an affine coordinate by
\[
  \textrm{cr}_S \colon \Mbar_{0,n} \longrightarrow \P^1, \qquad
  (C; x_1, \ldots, x_n) \longmapsto
  \frac{(x_a - x_c)(x_b - x_n)}{(x_a - x_n)(x_b - x_c)} ,
\]
and $\cL_S = \textrm{cr}_S^{*}\,\cO_{\P^1}(1)$.  We denote
\[
  \textrm{cr} \coloneqq (\textrm{cr}_S)_{|S| = 3} \colon \Mbar_{0,n} \longrightarrow
  \big(\P^1\big)^{\binom{n-1}{3}}
\]
for the product of all cross-ratio maps as $S$ ranges over all subsets of size $3$ in $[n-1]$, and
\[
\cO(y) \coloneqq \boxtimes_S \; \cO_{\P^1}(y_S), \qquad y \in \Z^{\binom{[n-1]}{3}}.
\]

\begin{definition}[zonotopal class on $\Mbar_{0,n}$]\label{def:m0n-zonotopal}
  A line bundle on $\Mbar_{0,n}$ is \emph{zonotopal}, if it is
  $\textrm{cr}^{*}\cO(y)$ for some $y \geqslant 0$, that is, if it is a pullback
  of a nef line bundle on $\big(\P^1\big)^{\binom{n-1}{3}}$ along $\textrm{cr}$.
  A $K$-class $\xi \in K(\Mbar_{0,n})$ is \emph{zonotopal} if it is represented by
  a zonotopal line bundle.
\end{definition}
The intersection numbers of the first Chern classes of these pullbacks along cross-ratio maps are studied as special cases of \emph{Kapranov degrees}. 
See \cite{brakensiek2025kapranov}, as well as an earlier work by Silversmith on cross-ratio degrees \cite{silversmith2022cross}. 

For $y \in \Z^{\binom{[n-1]}{3}}$,
\[
  \textrm{cr}^{*}\cO(y) \;=\; \bigotimes_{|S| = 3} \textrm{cr}_S^{*}\,\cO_{\P^1}(y_S)
  \;=\; \bigotimes_{|S| = 3} \cL_S^{\otimes y_S} \;=\; \cL_y .
\]
Therefore, the line bundle $\cL_y$ is zonotopal in $K(\Mbar_{0, n})$ if $y \geqslant 0$ and $y_S = 0$
for $|S| \geqslant 4$.

We now describe $\Mbar_{0, n}$ as the wonderful variety associated to the graphic matroid $\M(K_{n-1})$ of
the complete graph on $[n-1]$ of rank $n-2$ with the minimal building set. Its lattice of flats consists of the partitions of
$[n-1]$: a partition into $\ell$ nonempty parts has rank $n-1-\ell$, two partitions $\lambda \leqslant \lambda'$ if $\lambda$ refines $\lambda'$, and a flat is
\emph{connected}, if all but exactly one part are singletons. For
$S \subseteq [n-1]$ with $|S| \geqslant 2$, write $F_S$ for the connected flat whose
non-singleton part is $S$, and so $\rk(F_S) = |S| - 1$.  The
connected flats
\[
  \G_{\min} \;=\; \big\{\, F_S \;:\; S \subseteq [n-1],\ |S| \geqslant 2 \,\big\}
\]
form the \emph{minimal building set} of the braid arrangement, and $\Mbar_{0,n}$ is the
wonderful variety associated to $\M(K_{n-1})$ together with the minimal building
set $\G_{\min}$ \cite{feichtner2004chow, kapranov1993chow}.
There is a birational map, induced by projecting \eqref{eq:wonderful} onto the factors indexed by $\G_{\min}$
\[
  p_n \; \colon \; W_{\M(K_{n-1})} \longrightarrow \Mbar_{0,n} = W_{\M(K_{n-1}), \G_{\min}}
\]
which is an iterated blowup at smooth centers, the strict transforms of the
linear spaces indexed by the flats outside $\G_{\min}$ \cite[\S 4]{larson2024k}.
Pullback along $p$ takes a line bundle on $\Mbar_{0,n}$ to one on
$W_{\M(K_{n-1})}$, and $p^{*}\cL_S \;=\; \cL_{F_S}$,
the line bundle associated to the flat $F_S$ from \Cref{sec:geometry}.
A flat $F_S$ has rank $2$ in $\M(K_{n-1})$ if and only if $|S| = 3$, so
\Cref{def:m0n-zonotopal} is a translation of the condition in \Cref{def:zonotopal-class} that the
weight vanishes on the flats of rank at least $3$.

In \cite[\S 9]{larson2024k}, Larson, Li, Payne and Proudfoot derived Euler characteristics on $\Mbar_{0,n}$ using a similar condition to the dragon Hall--Rado condition.
A tuple $c = (c_S)$ of
nonnegative integers indexed by the subsets $S \subseteq [n-1]$ with
$|S| \geqslant 3$ is Cerberus if
for every nonzero $c' \leqslant c$,
\begin{equation}\label{eq:cerberus-body}
  \tag{Cerberus}
  \Big|\bigcup_{c'_S > 0} S \cup \{n\}\Big| \;\geqslant\; \sum_S c'_S + 3.
\end{equation}
We denote by $\Cer(n)$ the set of all such tuples.

\begin{proposition}[{\cite[Theorem 9.1]{larson2024k}}]
  \label{prop:LLPP-cerberus}
The Snapper polynomial $\chi\big(\Mbar_{0,n}, \cL_y^{\otimes q}\big)$ is
\begin{equation}\label{eq:cerberus}
  \chi\big(\Mbar_{0,n}, \cL_y^{\otimes q}\big) \;=\; \sum_{c \in \Cer(n)} \ \prod_{S} \binom{y_S\, q + c_S - 1}{c_S}.
\end{equation}
\end{proposition}

\begin{lemma}
\label{lem:compute-EC-as-pullback}
Let $\M = \M(K_{n-1})$ and let $y = \begin{pmatrix} y_S \end{pmatrix}$ be a tuple of nonnegative integers indexed by $S \subseteq [n-1]$ of size at least $3$. 
Let 
\[
P_y \;\coloneqq\; \sum_{S} \;y_S\; \; \Delta_{F_S}\; \subseteq \R^{E(\M)}, \qquad \Delta_{F_S} \; \coloneqq \;\textrm{conv}\Big(e_{uv} \mid (u, v) \in \binom{S}{2}\Big). 
\]
Then the following statements hold: 
\begin{enumerate}[leftmargin=20pt]
  \item $P_y$ is a lattice generalized permutohedron with 
  \[
  \omega_{\M, P_y}(F) = \begin{cases}
    y_S & \textrm{ if $F$ is a connected line in $\M$}, \\
    0 & \textrm{ otherwise}. 
  \end{cases}
  \]
  \item $p_n^{\ast}(\cL_{y}) = \cL_{P_y}$. 
  \item $\chi(\Mbar_{0, n}, \cL_y^{\otimes q}) = \chi_{\M, \cL_{P_y}}(q)$. 
\end{enumerate}
\end{lemma}

\begin{proof}
  The first statement follows from \Cref{lem:support-general}. 
  For the second statement, the definition of $\cL_{P_y}$ in \eqref{eq:LP-flats} implies that 
  \[
  \cL_{P_y} = \bigotimes_{\textrm{ nonminimal flat } F} \cL_F^{\otimes \omega_{\M, P_y}(F)} = \bigotimes_{S \subseteq [n-1], \abs{S} \geqslant 3} \cL_{F_S}^{\otimes y_S}, 
  \] where the second equality follows from (1). 
  Since $p_n^{\ast}$ is a ring homomorphism and $p^{\ast}\cL_S = \cL_{F_S}$, 
  \[
  p_n^{\ast}\cL_y = \bigotimes_{S \subseteq [n-1], \abs{S} \geqslant 3} (p_n^{\ast} \cL_S)^{\otimes y_S} = \bigotimes_{S \subseteq [n-1], \abs{S} \geqslant 3} \cL_{F_S}^{\otimes y_S} = \cL_{P_y}. 
  \] 
  For the last statement, since $p_n$ is an iterative sequence of blowups at smooth centers in smooth projective varieties, $R p_{n, \ast} \calO_{W_{\M}} = \calO_{\Mbar_{0, n}}$, 
  and the projective formula implies that $\chi(\Mbar_{0, n}, \xi) = \chi(W_{\M}, p^{\ast} \xi)$ for every $\xi \in K(\Mbar_{0, n})$. Applying this to $\xi = \cL_{y}^{\otimes q}$ with (2) and the isomorphism $K(\M) \cong K(W_{\M})$ in \cite[Proposition 1.13]{larson2024k} gives the result. 
\end{proof}

\begin{theorem}\label{thm:m0n-triples}
  If $\cL_y$ is zonotopal in $K(\Mbar_{0, n})$, and let $\DT_n$ denote the Dilworth truncation of $\M(K_{n-1})$ along its connected lines, 
  then the Snapper polynomial 
  \[
    \chi(\Mbar_{0, n}, \cL_y^{\otimes q}) \;=\; \sum_{I \in \DT_n} \Big(\prod_{S \in I} y_S\Big)\, q^{|I|} .
  \]
  As a consequence, $\chi(\Mbar_{0, n}, \cL_y^{\otimes q})$ is magic positive, and its magic vector is the $y$-weighted
  $h$-vector of the independence complex of $\DT_n$.
\end{theorem}

\begin{proof}
  This follows directly from \Cref{lem:compute-EC-as-pullback}(3), \Cref{thm:ec-zonotopal} and \Cref{lem:compute-EC-as-pullback}(1). 
\end{proof}

\begin{corollary}
  If $\cL_y$ is zonotopal in $K(\Mbar_{0, n})$, then $\hstar(\Mbar_{0,n}, \cL_y)$ is
  real-rooted.
\end{corollary}

\begin{example}\label{ex:m0n-small}
  We compute the Snapper polynomials of zonotopal classes with $y_S = 1$ for $S \in \binom{[n-1]}{3}$ and $n = 5, 6, 7,8$ as follows. 
  \[
    \chi\big(\Mbar_{0,n}, \cL_y^{\otimes q}\big) =
    \begin{cases}
      1 + 4q + 6q^2, & n = 5, \\
      1 + 10q + 45q^2 + 100q^3, & n = 6, \\
      1 + 20q + 190q^2 + 1080q^3 + 3360q^4, & n = 7, \\
      1 + 35q + 595q^2 + 6405q^3 + 46200q^4 + 191436q^5, & n = 8,
    \end{cases}
  \]
  with magic vectors
  \[
    m =
    \begin{cases}
      (1,2,3), & n = 5, \\
      (1,7,28,64), & n = 6, \\
      (1,16,136,756,2451), & n = 7, \\
      (1,30,465,4820,35040,151080), & n = 8,
    \end{cases}
  \]
  and $\hstar$-polynomials
  \[
    \hstar(t) =
    \begin{cases}
      1 + 8t + 3t^2, & n = 5, \\
      1 + 152t + 383t^2 + 64t^3, & n = 6, \\
      1 + 4646t + 39956t^2 + 33586t^3 + 2451t^4, & n = 7, \\
      1 + 244666t + 5450826t^2 + 12596546t^3 + 4529201t^4 + 151080t^5, & n = 8.
    \end{cases}
  \]
\end{example}

Beyond zonotopal classes, we consider line bundles $\cL_y$ in $K(\Mbar_{0, n})$ whose tuple $y$ is supported on a single $S$. 
\begin{proposition}\label{prop:m0n-single}
  Let $S \subseteq [n-1]$ with $|S| \geqslant 3$, and let $y\in\Z_{>0}$. Then
  \[
    \chi\big(\Mbar_{0,n}, \cL_S^{\otimes yq}\big)
    \;=\; \binom{yq + |S| - 2}{|S| - 2} ,
  \]
  and $\cL_S^{\otimes y}$ is magic positive if and only if $y \geqslant |S| - 2$.
  In particular, the Euler characteristic of $\cL_{[n-1]}^{\otimes q}$ is not magic
  positive for $n \geqslant 5$.
\end{proposition}

\begin{proof}
  By \eqref{eq:cerberus}, $c = c_S$ is Cerberus if and only if $0 \leqslant c_S \leqslant \abs{S} - 2$.  
  By \Cref{prop:LLPP-cerberus},
  \[
  \chi\big(\Mbar_{0,n},\cL_S^{\otimes yq}\big) = \sum_{0 \leqslant c_S \leqslant \abs{S} - 2} \binom{y_S \, q + c_S - 1}{c_S} = \binom{y_S \,q + \abs{S} -2}{\abs{S} -2}, 
  \] 
  where the last equality is given by the hockey-stick identity. 
  For $y>0$, its magic transform is
  \[
  m(\chi\big(\Mbar_{0,n},\cL_S^{\otimes yq}\big); t) = \frac1{(|S|-2)!}\prod_{j=1}^{|S|-2}(j+(y-j)t).
  \] 
  Since the coefficient of $t$ is nonnegative if and only if $y \geqslant \abs{S} - 2$, we have that 
  $\chi\big(\Mbar_{0,n},\cL_S^{\otimes yq}\big)$ is magic positive if and only if $y \geqslant \abs{S} - 2$. 
\end{proof}

\begin{remark}[matroids with arbitrary building sets]\label{rem:building-sets}
  The techniques for the above results do not depend on the minimality of the building sets, or the braid matroids. Similar techniques above hold for a matroid with any building set
  $\G$ and a line bundle $\bigotimes_{F\in\G,\,\rk F=2}\cL_F^{\otimes y_F}$
  with $y_F\geqslant0$. We do not pursue the details of this direction in the present paper. 
\end{remark}

We now consider the cotangent line bundles $\cL_i$ with $c_1(\cL_i) = \psi_i$ on $\Mbar_{0, n}$.
Let $d \coloneqq n - 3$ be the dimension of
$\Mbar_{0,n}$.  For $n \geqslant 4$, \Cref{prop:m0n-single} implies
\[
  \chi\big(\Mbar_{0,n}, \cL_i^{\otimes q}\big) \;=\; \binom{q+d}{d},
\]
which is not magic positive for $n \geqslant 5$, so $\cL_i$ is not zonotopal for
$n \geqslant 5$ by \Cref{thm:m0n-triples}.

\begin{theorem}\label{thm:psi-two}
  For distinct $i, j \in [n]$,
  \begin{align*}
    \chi\big(\Mbar_{0,n}, (\cL_i \otimes \cL_j)^{\otimes q}\big)
    &\;=\; \sum_{k=0}^{d} \binom{d}{k}^{2} \binom{q + d - k}{d}, \\
    \hstar\big(\Mbar_{0,n}, \cL_i \otimes \cL_j\big)
    &\;=\; \sum_{k=0}^{d} \binom{d}{k}^{2} t^{k} .
  \end{align*}
  The $\hstar$-polynomial is real-rooted.
\end{theorem}

\begin{proof}
  The case $d=0$ is immediate. Suppose $d \geqslant1$ and without loss of generality let $i = 1, j = 2$. By 
  \cite[Theorem 9.2]{larson2024k}, and also Lee, and Pandharipande
  \cite{lee1997formula,pandharipande1997symmetric},
  \begin{equation}\label{eq:llpp-psi}
    \chi\Big(\Mbar_{0,n},\bigotimes_{k=1}^n\cL_k^{\otimes w_k q}\Big)
    =\sum_{|a|\leqslant d}\binom{d}{a_1,\ldots,a_n,d-|a|}
      \prod_{k=1}^n\binom{w_k q}{a_k}.
  \end{equation}
  Setting $w=e_1+e_2$, summing over the second coordinate by Vandermonde's
  identity, 
  \begin{align*}
    \chi\Big(\Mbar_{0,n},\bigotimes_{k=1}^n\cL_k^{\otimes w_k q}\Big) &=\sum_{a=0}^d\binom da\binom qa\binom{q+d-a}{d-a}\\
        &=\sum_{a=0}^d\binom da^2\binom{q+d-a}{d}.
  \end{align*}
  The definition of the $\hstar$-vector in \eqref{eq:Wf-change-of-basis} yields the second formula.

  The formula for the $h^{\ast}$-polynomial identifies it with the $d$th type-B Narayana polynomial, whose real-rootedness follows from a result of orthogonal polynomials \cite{szego1975orthogonal}. If $P_d$ is the Legendre polynomial of degree $d$, then
  \[
    \sum_{k=0}^d\binom dk^2t^k
      =(1-t)^d P_d\!\left(\frac{1+t}{1-t}\right).
  \]
  The $d$ zeros of $P_d$ lie in $(-1,1)$, and so
  their inverse images under this change of variable are real and negative.
\end{proof}

\begin{proposition}
  Let $\cL_i$ be the line bundle with $c_1(\cL_i) = \psi_i$, and let $i \ne j \in [n]$.  
  The Snapper polynomial  
  $\chi\big(\Mbar_{0,n}, (\cL_i \otimes \cL_j)^{\otimes q}\big)$ is magic positive
  if and only if $n \leqslant 7$. 
\end{proposition}

\begin{proof}
  Let $f_n(q) $ denote $\chi\big(\Mbar_{0,n}, (\cL_i \otimes \cL_j)^{\otimes q}\big)$. 
  By \Cref{thm:psi-two}, $f_n(0)=1$. By \Cref{cor:magic-extremes}, the first magic coefficient is $m_1=f_n'(0)-d$.
  Differentiating the binomial expression yields
  \[
    m_1=H_d+\sum_{k=1}^d\frac{(-1)^{k-1}}k\binom dk-d=2H_d-d,
    \qquad H_d=\sum_{j=1}^d\frac1j.
  \]
  The function $2H_d - d$ is already negative when $d = 5$, since $2H_5-5=\frac{-13}{30}$. 
  Furthermore, it is monotonously decreasing, since for any $d \geqslant2$,
  \[
  2H_{d+1}-(d+1)-(2H_d-d)=\frac{2}{d+1}-1<0,
  \]
  which implies that magic positivity of $f_n(q)$ fails for $d\geqslant5$.
  For $d=0,1,2,3,4$, respectively, the magic vectors are
  \[
    (1),\quad(1,1),\quad(1,1,1),\quad
    (1,\tfrac23,\tfrac23,1),\quad
    (1,\tfrac16,\tfrac7{12},\tfrac16,1).
  \]
  This proves our assertion.
\end{proof}

\section{Saturated and weakly saturated classes}
\label{sec:beyond}

In this section, we investigate beyond the zonotopal classes.  If $P$ is an arbitrary lattice
generalized permutohedron, \Cref{thm:dhr} shows that $\chi_{\M,P}$ is a sum over dragon-Hall--Rado tuples, which form a discrete polymatroid $\DHR(\M)$. We decompose $\DHR(\M)$ by a lexicographic shelling. For background on
polymatroids we refer to \cite{edmonds1970submodular}. Let $\DHR(\M, P)$ denote the dragon-Hall--Rado tuples supported on the support of $\omega_{\M, P}$: 
\[
\DHR(\M, P) \; \coloneqq \; \Big\{k \mid \rk_{\M}\big(\bigcup_{F \in J} F\big)  \geqslant \sum_{F \in J} k_F  + 1, \quad \varnothing \ne J \subseteq \supp(\omega_{\M, P})\Big\}
\]

\begin{proposition}\label{prop:polymatroid}
  The set $\DHR(\M, P)$ is a discrete polymatroid.
\end{proposition}

\begin{proof}
  This follows directly from \Cref{thm:edmonds} by using submodular and monotone function 
  \[
  f(J) = \rk_{\M}(\bigcup_{F \in J} F) - 1, 
  \] with the convention that $f(\varnothing) = -1$. 
\end{proof}

For $F \in \supp(\omega)$, recall the capacity and demand are defined as 
\[
  \capa_F \;=\; \rk_{\M}(F) - 1 \;=\; \max_{k \in \DHR(\M, P)} k_F,
  \qquad
  \dem_F \;=\; \min_{B \in \mathcal{B}} B_F ,
\]
where $\mathcal{B}$ is the set of maximal elements of $\DHR(\M, P)$ and $B_F$ is the multiplicity of $F$ in the top-degree monomial indexed by $B$.

\begin{lemma}\label{lem:single-flat}
  If $\omega_{\M, P}$ is supported on a single flat $F$, with capacity
  $\capa_F = \rk F - 1 \geqslant 1$ and $\omega = \omega_{\M, P}(F) > 0$, then $d = \capa_F$,
  \[
  \chi_{\M,P}(q) = \binom{\omega q + \capa_F}{\capa_F}, 
  \qquad 
  m(\chi_{\M,P}; y) \;=\; \frac{1}{\capa_F!}\ \prod_{j=1}^{\capa_F}\big(\, j + (\omega - j)\, y \,\big).
  \]
  In particular, $\chi_{M, P}(q)$ is magic positive, if and only if $\omega \geqslant \capa_F$.
\end{lemma}

\begin{proof}
  By \Cref{thm:dhr}, and the hockey-stick identity,
  \[
    \chi_{\M,P}(q) \;=\; \sum_{k = 0}^{\capa_F} \binom{\omega q + k - 1}{k}
    \;=\; \binom{\omega q + \capa_F}{\capa_F}
    \;=\; \frac{1}{\capa_F!}\prod_{j=1}^{\capa_F} (\omega q + j) .
  \]
  Since this is a product of $\capa_F$ linear factors in $q$, $d = \capa_F$.

  The second and the third identities follow from applying the magic transform. 
\end{proof}

\begin{example}
  If $\M$ is a Boolean matroid and $P=k\Delta_F$ with $|F|=d+1$ and $k\geqslant1$,
  then \Cref{lem:single-flat} implies that $P$ is magic positive if and only if
  $k\geqslant d$. Therefore, a simplex of dimension at least $2$ need not be
  magic positive, for instance $P = 2\Delta_{\{1,2,3,4\}}$.
  Since $\chi_{\M,k\Delta_F}(q)=\chi_{\M,\Delta_F}(kq)$, this is an instance of a
  theorem of Konoike \cite{konoike2025magic}: for any polynomial $f$ with positive
  real coefficients there is a threshold $k_0$ such that $f(kq)$ is magic positive
  for all $k\geqslant k_0$.  \Cref{lem:single-flat} identifies the threshold
  $k_0=d$ for $f(q)=\binom{q+d}{d}$.  Konoike also shows that no uniform bound is
  possible: for every $d\geqslant3$ and every $k$, there is a $d$-dimensional
  polytope whose $k$th dilation has Ehrhart polynomial that is not magic positive.
\end{example}

We next decompose $\DHR(\M, P)$ into
disjoint boxes. This allows us to write
$m(\chi_{\M,P}; y)$ as a sum of products, one factor for each flat of $\supp(\omega)$, and saturation ensures that every such factor has nonnegative coefficients.
By \cite{herzog2002discrete}, $\DHR(\M, P)$ satisfies the following properties. 
\begin{enumerate}[leftmargin=20pt]
    \item \emph{Hereditary}: If $k \in \DHR(\M, P)$ and
      $0 \leqslant k' \leqslant k$ coordinatewise, then $k' \in \DHR(\M)$.
    \item \emph{Basis exchange}: If $B$ and $B'$ are bases of $\DHR(\M, P)$ and
      $B_F > B'_F$ for a coordinate $F$, then $B - e_F + e_{F'}$ is a basis for some
      coordinate $F'$ with $B_{F'} < B'_{F'}$.
\end{enumerate}

\begin{lemma}\label{lem:lex-shelling}
  Let the bases of $\DHR(\M, P)$ be in increasing lexicographic order
  $B_1 < \cdots < B_N$, and for each $i$, let
  \[
    C_i \coloneqq \big\{ F : (B_i)_F > 0 \text{ and } B_i - e_F \leqslant B_j \text{ for some } j < i \big\},
  \]
  together with
  \[
    (R_i)_F \coloneqq \begin{cases} (B_i)_F & F \in C_i, \\ 0 & F \notin C_i. \end{cases}
  \]
  Then $\DHR(\M, P)$ is the disjoint union of
  $[R_i, B_i] = \{k : R_i \leqslant k \leqslant B_i\}$ for $1 \leqslant i \leqslant N$.
\end{lemma}

\begin{proof}
  By the hereditary property, it suffices to prove that
  \[
    [0, B_i] \cap \bigcup_{j < i} [0, B_j] \;=\; \bigcup_{F \in C_i} [0, B_i - e_F],
  \]
  The right-hand side is contained in the left-hand side by the definition of $C_i$.  For
  the other direction, let $k \leqslant B_i$ and $k \leqslant B_j$ with $j < i$, and let $F$ be the
  first coordinate in which $B_j$ and $B_i$ differ, and so $(B_j)_F < (B_i)_F$.
  The exchange property applied to $B_i$, $B_j$ and $F$ implies $F'$ with
  $(B_i)_{F'} < (B_j)_{F'}$ such that $B' = B_i - e_F + e_{F'}$ is a basis.  As
  $B_i$ and $B_j$ agree in all coordinates before $F$, we have $F' > F$, hence $B' < B_i$,
  i.e., $B' = B_{j'}$ for some $j' < i$. Since $B_i - e_F \leqslant B'$, this implies
  $F \in C_i$. Finally, $k_F \leqslant (B_j)_F \leqslant (B_i)_F - 1$, so $k \in [0, B_i - e_F]$.
\end{proof}

For integers $\omega, k, b \geqslant 0$, we define 
\begin{align*}
  \phi_{\omega,k}(y) \;\coloneqq\; \frac{1}{k!}\prod_{j=0}^{k-1}\big(j + (\omega-j)\,y\big), \qquad
  \Phi_{\omega, b}(y) \;\coloneqq\; \frac{1}{b!}\prod_{j=1}^{b}\big(j + (\omega-j)\,y\big).
\end{align*}

\begin{lemma}\label{lem:box-factors}
  For every $b \geqslant 0$,
  \[
  \sum_{k=0}^{b} (1-y)^{b-k}\,\phi_{\omega,k}(y) = \Phi_{\omega, b}(y).
  \]
  Furthermore, if $\omega \geqslant b$, then $\Phi_{\omega,b}(y)$ and $\phi_{\omega,b}$ have nonnegative coefficients.
\end{lemma}

\begin{theorem}\label{thm:sufficient}
  If $\omega_{\M, P} \geqslant 0$ and $\omega_{\M, P}(F) \geqslant \capa_F$ for every flat $F$ with $\omega_{\M, P}(F) > 0$, then $\chi_{\M, P}(q)$ is magic
  positive.  Furthermore, the magic transform of $\chi_{\M,P}$ is 
  \begin{equation}\label{eq:box-formula}
    m(\chi_{\M,P}; y) \;=\; \sum_{i=1}^{N}\ \prod_{F \notin C_i}\Phi_{\omega_F, (B_i)_F}(y)\
    \prod_{F \in C_i}\phi_{\omega_F,(B_i)_F}(y),
  \end{equation}
  and has nonnegative coefficients.
\end{theorem}

\begin{proof}
  Let $d = \deg \chi_{\M, P}(q)$. By \eqref{eq:dhr} and the substitution $q = y/(1-y)$,
  \[
    m(\chi_{\M,P}; y) = (1-y)^d\,\chi_{\M,P}\Big(\frac{y}{1-y}\Big)
         = \sum_{k \in \DHR(\M, P)} (1-y)^{d - |k|}\prod_F \phi_{\omega_F, k_F}(y).
  \]
  We split the sum according to the boxes of \Cref{lem:lex-shelling}. On any box $[R_i, B_i]$, one
  has
  \[d - |k| = \sum_F\big((B_i)_F - k_F\big). 
  \] Summing over the box factors yields
  \[
  \prod_F \sum_{k_F = (R_i)_F}^{(B_i)_F}(1-y)^{(B_i)_F - k_F} \, \phi_{\omega_F,k_F}(y).
  \]
  For $F \in C_i$, the inner sum is the single term $\phi_{\omega_F,(B_i)_F}(y)$. For
  $F \notin C_i$, the inner sum is $\Phi_{\omega_F, (B_i)_F}(y)$ by \Cref{lem:box-factors}.
  Then \Cref{lem:box-factors} together with $(B_i)_F \leqslant \capa_F \leqslant \omega_F$ implies that each factor has nonnegative coefficients.
\end{proof}

For zonotopal $P$, every capacity is $1$, with $\Phi_{\omega, 1}(y) = 1 + (\omega-1)y$ and
$\phi_{\omega,1} = \omega y$, so \eqref{eq:box-formula} recovers the formula of
\Cref{thm:magic-zonotopal} and \Cref{lem:lex-shelling} is a lexicographic shelling of the Dilworth truncation
$\DT(\M)$. 

\begin{lemma}\label{lem:forced}
  For every flat $F \in \supp(\omega_{\M, P})$, the polynomial
  \[
  \binom{\omega_F q + \dem_F}{\dem_F}
  \] is a factor of $\chi_{\M,P}(q)$, and the polynomial $\Phi_{\omega_F, \dem_F}(y)$ is a factor of $m(\chi_{\M,P}; y)$.
\end{lemma}

\begin{proof}
  This follows from \Cref{lem:box-factors} and \Cref{thm:dhr}. 
\end{proof}

\begin{theorem}\label{thm:necessary}
  If $\chi_{\M,P}(q)$ is magic positive, then $\omega_F \geqslant \dem_F$, for every flat $F$ with
  $\omega_F > 0$, that is, $P$ is a weakly saturated class in $K(\M)$.
\end{theorem}

\begin{proof}
  Suppose for contradiction that $\omega_F < \dem_F$. Then the factor $j + (\omega_F - j)y$ of $\Phi_{\omega_F, \dem_F}(y)$
  with $j = \dem_F$ vanishes for $y = \dem_F/(\dem_F - \omega_F) > 0$.  By
  \Cref{lem:forced}, $\Phi_{\omega_F, \dem_F}(y)$ divides $m(\chi_{\M,P}; y)$, the
  magic transform has a positive real root.  But a polynomial with nonnegative
  coefficients and constant term $m_0 = 1$ is at least $1$ on $[0, \infty)$.
\end{proof}

\begin{corollary}\label{cor:box}
  The following are equivalent:
  \begin{enumerate}[leftmargin=20pt]
    \item For every nonempty set $J$ of flats in $\supp(\omega_{\M,P})$,
      \[
        \sum_{F \in J} \capa_F \;\leqslant\; \rk_\M\Big(\bigcup_{F \in J} F\Big) - 1 ;
      \]
    \item 
    \[
    \DHR(\M,P) = \prod_{F \in \supp(\omega_{\M,P})} \big([0, \capa_F] \cap \Z\big).
    \]
  \end{enumerate}
  In this case, $\dem_F = \capa_F$ for every $F$, the saturated and the weakly
  saturated conditions coincide, and $\chi_{\M,P}(q)$ is magic positive if and only
  if $\omega_F \geqslant \capa_F$ for every $F \in \supp(\omega_{\M,P})$.
\end{corollary}

\Cref{thm:sufficient,thm:necessary} implies the following relationships 
\begin{equation}\label{eq:sandwich}
  \text{zonotopal} \;\Longrightarrow\; \text{saturated} \;\Longrightarrow\;
  \text{magic positive} \;\Longrightarrow\; \text{weakly saturated}.
\end{equation}

\begin{example}
\label{ex:strict-implications}
The above implications are strict. When $\M = \B_3$, the class given by 
$P= 2\Delta_{123}$ is saturated but not zonotopal. When $\M = \B_4$, the class given by
$P=\Delta_{123}+\Delta_{124}$ is not saturated, and yet
$\chi_{\B_4,P}(q)=(1+q)^3$, so it is magic positive. 
When $\M = \B_5$, the polytope
$P=\Delta_{1234}+\Delta_{1235}$ is weakly saturated, but direct evaluation of \eqref{eq:dhr} gives
\[
  m(\chi_{\B_5,P};y)=1-\frac13y-\frac1{12}y^2,
\]
so it is not magic positive.
We do not yet know a complete characterization of magic positivity of $\chi_{\M, P}$ using capacity and demand. 
\end{example}

Ferroni and Higashitani ask in \cite[Problem 4.23]{ferroni2024examples} which
families of $Y$-generalized permutohedra have magic positive Ehrhart polynomials. 
In the special case of Boolean matroids, \Cref{cor:ygp} recovers \cite[Theorem 3.19]{avila2026luck} using a weaker hypothesis for achieving magic positivity (but do not give the combinatorial interpretation in terms of lucky cars) and
\Cref{thm:sufficient} is its extension to an arbitrary loopless matroid. 

\begin{corollary}\label{cor:ygp}
  If $y_S \geqslant |S| - 1$ for every $S$ with $y_S > 0$, then the $Y$-generalized
  permutohedron $P = \sum_S y_S\Delta_S$ is magic
  positive, Ehrhart-positive, and its $\hstar$-polynomial is real-rooted.
\end{corollary}

We turn back to $\Mbar_{0,n}$ in
\Cref{sec:m0n} to obtain a similar sufficient condition. 
\begin{corollary}\label{thm:m0n-sufficient}
  Let $y \geqslant 0$.  If $y_S \geqslant |S| - 2$ for every $S$ with $y_S > 0$,
  then $\chi\big(\Mbar_{0,n}, \cL_y^{\otimes q}\big)$ is magic positive.
\end{corollary}

\begin{proof}
  Application of \Cref{lem:lex-shelling,lem:box-factors} to the Dilworth truncation of $\M(K_{n-1})$ along its connected lines, together with \Cref{thm:sufficient} implies the result. 
\end{proof}

\begin{remark}
  The proof of \Cref{thm:necessary} also gives the necessary inequalities
  $y_S\geqslant\min_B B_S$, where $B$ ranges over the bases of
  the polymatroid $\Cer(n)$ along $\supp(y)$.
\end{remark}

\bibliographystyle{amsalpha}
\bibliography{hstar-refs}

\end{document}

%% file: parallelohedra.tex
\newcommand{\zonoPath}{%
\begin{tikzpicture}[x=1cm,y=1cm,line join=round,baseline=(current bounding box.center)]
  \draw[gray!45,dashed,line width=0.3pt] (0.057,0.220) -- (-0.076,-0.963);
  \draw[gray!45,dashed,line width=0.3pt] (0.057,0.220) -- (-0.867,0.963);
  \draw[gray!45,dashed,line width=0.3pt] (0.057,0.220) -- (1.000,0.220);
  \filldraw[fill=gA!32,draw=gA!80!black,line width=0.5pt] (-0.057,-0.220) -- (0.867,-0.963) -- (1.000,0.220) -- (0.076,0.963) -- cycle;
  \filldraw[fill=gA!42,draw=gA!80!black,line width=0.5pt] (-1.000,-0.220) -- (-0.057,-0.220) -- (0.076,0.963) -- (-0.867,0.963) -- cycle;
  \filldraw[fill=gA!40,draw=gA!80!black,line width=0.5pt] (-0.076,-0.963) -- (0.867,-0.963) -- (-0.057,-0.220) -- (-1.000,-0.220) -- cycle;
\end{tikzpicture}}

\newcommand{\zonoPaw}{%
\begin{tikzpicture}[x=1cm,y=1cm,line join=round,baseline=(current bounding box.center)]
  \draw[gray!45,dashed,line width=0.3pt] (0.035,-0.113) -- (-0.064,-1.000);
  \draw[gray!45,dashed,line width=0.3pt] (0.035,-0.113) -- (-0.657,0.443);
  \draw[gray!45,dashed,line width=0.3pt] (0.035,-0.113) -- (0.742,-0.113);
  \draw[gray!45,dashed,line width=0.3pt] (-0.657,0.443) -- (-0.757,-0.443);
  \draw[gray!45,dashed,line width=0.3pt] (-0.657,0.443) -- (-0.643,1.000);
  \filldraw[fill=gA!14,draw=gA!80!black,line width=0.5pt] (0.657,-0.443) -- (0.643,-1.000) -- (0.742,-0.113) -- (0.757,0.443) -- cycle;
  \filldraw[fill=gA!32,draw=gA!80!black,line width=0.5pt] (-0.035,0.113) -- (0.657,-0.443) -- (0.757,0.443) -- (0.064,1.000) -- cycle;
  \filldraw[fill=gA!42,draw=gA!80!black,line width=0.5pt] (-0.742,0.113) -- (-0.035,0.113) -- (0.064,1.000) -- (-0.643,1.000) -- cycle;
  \filldraw[fill=gA!40,draw=gA!80!black,line width=0.5pt] (-0.064,-1.000) -- (0.643,-1.000) -- (0.657,-0.443) -- (-0.035,0.113) -- (-0.742,0.113) -- (-0.757,-0.443) -- cycle;
\end{tikzpicture}}

\newcommand{\zonoCycle}{%
\begin{tikzpicture}[x=1cm,y=1cm,line join=round,baseline=(current bounding box.center)]
  \draw[gray!45,dashed,line width=0.3pt] (0.096,0.372) -- (-0.016,-0.628);
  \draw[gray!45,dashed,line width=0.3pt] (0.096,0.372) -- (-0.685,1.000);
  \draw[gray!45,dashed,line width=0.3pt] (0.096,0.372) -- (0.893,0.372);
  \draw[gray!45,dashed,line width=0.3pt] (-0.016,-0.628) -- (-0.797,0.000);
  \draw[gray!45,dashed,line width=0.3pt] (-0.016,-0.628) -- (0.781,-0.628);
  \draw[gray!45,dashed,line width=0.3pt] (-0.016,-0.628) -- (-0.112,-1.000);
  \draw[gray!45,dashed,line width=0.3pt] (-0.685,1.000) -- (-0.797,0.000);
  \draw[gray!45,dashed,line width=0.3pt] (-0.797,0.000) -- (-0.893,-0.372);
  \filldraw[fill=gA!31,draw=gA!80!black,line width=0.5pt] (-0.781,0.628) -- (0.016,0.628) -- (0.112,1.000) -- (-0.685,1.000) -- cycle;
  \filldraw[fill=gA!20,draw=gA!80!black,line width=0.5pt] (0.797,0.000) -- (0.893,0.372) -- (0.112,1.000) -- (0.016,0.628) -- cycle;
  \filldraw[fill=gA!14,draw=gA!80!black,line width=0.5pt] (0.685,-1.000) -- (0.781,-0.628) -- (0.893,0.372) -- (0.797,0.000) -- cycle;
  \filldraw[fill=gA!42,draw=gA!80!black,line width=0.5pt] (-0.893,-0.372) -- (-0.096,-0.372) -- (0.016,0.628) -- (-0.781,0.628) -- cycle;
  \filldraw[fill=gA!32,draw=gA!80!black,line width=0.5pt] (-0.096,-0.372) -- (0.685,-1.000) -- (0.797,0.000) -- (0.016,0.628) -- cycle;
  \filldraw[fill=gA!40,draw=gA!80!black,line width=0.5pt] (-0.112,-1.000) -- (0.685,-1.000) -- (-0.096,-0.372) -- (-0.893,-0.372) -- cycle;
\end{tikzpicture}}

\newcommand{\zonoDiamond}{%
\begin{tikzpicture}[x=1cm,y=1cm,line join=round,baseline=(current bounding box.center)]
  \draw[gray!45,dashed,line width=0.3pt] (0.370,-0.198) -- (-0.235,0.288);
  \draw[gray!45,dashed,line width=0.3pt] (0.370,-0.198) -- (-0.321,-0.486);
  \draw[gray!45,dashed,line width=0.3pt] (0.370,-0.198) -- (0.987,-0.198);
  \draw[gray!45,dashed,line width=0.3pt] (-0.235,0.288) -- (-0.926,0.000);
  \draw[gray!45,dashed,line width=0.3pt] (-0.235,0.288) -- (-0.222,0.774);
  \draw[gray!45,dashed,line width=0.3pt] (-0.321,-0.486) -- (-0.926,0.000);
  \draw[gray!45,dashed,line width=0.3pt] (-0.321,-0.486) -- (-0.395,-0.774);
  \draw[gray!45,dashed,line width=0.3pt] (-0.926,0.000) -- (-0.913,0.486);
  \draw[gray!45,dashed,line width=0.3pt] (-0.926,0.000) -- (-1.000,-0.288);
  \filldraw[fill=gA!20,draw=gA!80!black,line width=0.5pt] (0.926,0.000) -- (1.000,0.288) -- (0.395,0.774) -- (0.321,0.486) -- cycle;
  \filldraw[fill=gA!14,draw=gA!80!black,line width=0.5pt] (0.913,-0.486) -- (0.987,-0.198) -- (1.000,0.288) -- (0.926,0.000) -- cycle;
  \filldraw[fill=gA!27,draw=gA!80!black,line width=0.5pt] (0.222,-0.774) -- (0.913,-0.486) -- (0.926,0.000) -- (0.235,-0.288) -- cycle;
  \filldraw[fill=gA!32,draw=gA!80!black,line width=0.5pt] (0.235,-0.288) -- (0.926,0.000) -- (0.321,0.486) -- (-0.370,0.198) -- cycle;
  \filldraw[fill=gA!31,draw=gA!80!black,line width=0.5pt] (-0.913,0.486) -- (-0.987,0.198) -- (-0.370,0.198) -- (0.321,0.486) -- (0.395,0.774) -- (-0.222,0.774) -- cycle;
  \filldraw[fill=gA!40,draw=gA!80!black,line width=0.5pt] (-0.395,-0.774) -- (0.222,-0.774) -- (0.235,-0.288) -- (-0.370,0.198) -- (-0.987,0.198) -- (-1.000,-0.288) -- cycle;
\end{tikzpicture}}

\newcommand{\zonoComplete}{%
\begin{tikzpicture}[x=1cm,y=1cm,line join=round,baseline=(current bounding box.center)]
  \draw[gray!45,dashed,line width=0.3pt] (0.356,0.163) -- (0.281,-0.504);
  \draw[gray!45,dashed,line width=0.3pt] (0.356,0.163) -- (-0.165,0.581);
  \draw[gray!45,dashed,line width=0.3pt] (0.356,0.163) -- (0.888,0.163);
  \draw[gray!45,dashed,line width=0.3pt] (0.281,-0.504) -- (-0.314,-0.752);
  \draw[gray!45,dashed,line width=0.3pt] (0.281,-0.504) -- (0.813,-0.504);
  \draw[gray!45,dashed,line width=0.3pt] (-0.165,0.581) -- (-0.760,0.333);
  \draw[gray!45,dashed,line width=0.3pt] (-0.165,0.581) -- (-0.154,1.000);
  \draw[gray!45,dashed,line width=0.3pt] (-0.314,-0.752) -- (-0.835,-0.333);
  \draw[gray!45,dashed,line width=0.3pt] (-0.314,-0.752) -- (-0.378,-1.000);
  \draw[gray!45,dashed,line width=0.3pt] (-0.760,0.333) -- (-0.835,-0.333);
  \draw[gray!45,dashed,line width=0.3pt] (-0.760,0.333) -- (-0.749,0.752);
  \draw[gray!45,dashed,line width=0.3pt] (-0.835,-0.333) -- (-0.899,-0.581);
  \filldraw[fill=gA!31,draw=gA!80!black,line width=0.5pt] (-0.749,0.752) -- (-0.813,0.504) -- (-0.281,0.504) -- (0.314,0.752) -- (0.378,1.000) -- (-0.154,1.000) -- cycle;
  \filldraw[fill=gA!20,draw=gA!80!black,line width=0.5pt] (0.835,0.333) -- (0.899,0.581) -- (0.378,1.000) -- (0.314,0.752) -- cycle;
  \filldraw[fill=gA!14,draw=gA!80!black,line width=0.5pt] (0.760,-0.333) -- (0.749,-0.752) -- (0.813,-0.504) -- (0.888,0.163) -- (0.899,0.581) -- (0.835,0.333) -- cycle;
  \filldraw[fill=gA!27,draw=gA!80!black,line width=0.5pt] (0.154,-1.000) -- (0.749,-0.752) -- (0.760,-0.333) -- (0.165,-0.581) -- cycle;
  \filldraw[fill=gA!42,draw=gA!80!black,line width=0.5pt] (-0.888,-0.163) -- (-0.356,-0.163) -- (-0.281,0.504) -- (-0.813,0.504) -- cycle;
  \filldraw[fill=gA!40,draw=gA!80!black,line width=0.5pt] (-0.378,-1.000) -- (0.154,-1.000) -- (0.165,-0.581) -- (-0.356,-0.163) -- (-0.888,-0.163) -- (-0.899,-0.581) -- cycle;
  \filldraw[fill=gA!32,draw=gA!80!black,line width=0.5pt] (-0.356,-0.163) -- (0.165,-0.581) -- (0.760,-0.333) -- (0.835,0.333) -- (0.314,0.752) -- (-0.281,0.504) -- cycle;
\end{tikzpicture}}

%% file: hstar-refs.bib
@article{larson2024k,
    author = {Larson, Matt and Li, Shiyue and Payne, Sam and Proudfoot, Nicholas},
    title = {{$K$}-rings of wonderful varieties and matroids},
    journal = {Adv. Math.},
    volume = {441},
    year = {2024},
    pages = {Paper No. 109554},
    doi = {10.1016/j.aim.2024.109554}
}

@article{ardila2010matroid,
    author = {Ardila, Federico and Benedetti, Carolina and Doker, Jeffrey},
    title = {Matroid polytopes and their volumes},
    journal = {Discrete Comput. Geom.},
    volume = {43},
    number = {4},
    year = {2010},
    pages = {841--854},
    doi = {10.1007/s00454-009-9232-9}
}

@article{backman2024simplicial,
    author = {Backman, Spencer and Eur, Christopher and Simpson, Connor},
    title = {Simplicial generation of {C}how rings of matroids},
    journal = {J. Eur. Math. Soc. (JEMS)},
    volume = {26},
    year = {2024},
    number = {11},
    pages = {4491--4535}
}

@article{beck2019h,
    author = {Beck, Matthias and Jochemko, Katharina and McCullough, Emily},
    title = {$h^*$-polynomials of zonotopes},
    journal = {Trans. Amer. Math. Soc.},
    volume = {371},
    year = {2019},
    pages = {2021--2042}
}

@misc{beck2026matroidal,
    author = {Beck, Matthias and Klivans, Caroline and Ross, Dustin},
    title = {A matroidal twist on a formula of {B}rion},
    year = {2026},
    note = {arXiv:2604.15253}
}

@book{beck2015computing,
    author = {Beck, Matthias and Robins, Sinai},
    title = {Computing the Continuous Discretely: Integer-Point Enumeration in Polyhedra},
    series = {Undergraduate Texts in Mathematics},
    edition = {Second},
    publisher = {Springer, New York},
    year = {2015}
}

@article{branden2006linear,
    author = {Br{\"a}nd{\'e}n, Petter},
    title = {On linear transformations preserving the {P}\'olya frequency property},
    journal = {Trans. Amer. Math. Soc.},
    volume = {358},
    number = {8},
    year = {2006},
    pages = {3697--3716}
}

@article{brenti1994q,
    author = {Brenti, Francesco},
    title = {$q$-{E}ulerian polynomials arising from {C}oxeter groups},
    journal = {European J. Combin.},
    volume = {15},
    year = {1994},
    pages = {417--441}
}

@incollection{brion2003multiplicity,
    author = {Brion, Michel},
    title = {Multiplicity-free subvarieties of flag varieties},
    booktitle = {Commutative algebra ({G}renoble/{L}yon, 2001)},
    series = {Contemp. Math.},
    volume = {331},
    publisher = {Amer. Math. Soc.},
    year = {2003},
    pages = {13--23}
}

@article{castillo2020when,
    author = {Castillo, Federico and Cid-Ruiz, Yairon and Li, Binglin and Monta{\~n}o, Jonathan and Zhang, Naizhen},
    title = {When are multidegrees positive?},
    journal = {Adv. Math.},
    volume = {374},
    year = {2020},
    pages = {Paper No. 107382, 34}
}

@book{cox2011toric,
    author = {Cox, David A. and Little, John B. and Schenck, Henry K.},
    title = {Toric Varieties},
    series = {Graduate Studies in Mathematics},
    volume = {124},
    publisher = {American Mathematical Society, Providence, RI},
    year = {2011}
}

@article{deconcini1995wonderful,
    author = {De Concini, C. and Procesi, C.},
    title = {Wonderful models of subspace arrangements},
    journal = {Selecta Math. (N.S.)},
    volume = {1},
    number = {3},
    year = {1995},
    pages = {459--494}
}

@incollection{edmonds1970submodular,
    author = {Edmonds, Jack},
    title = {Submodular functions, matroids, and certain polyhedra},
    booktitle = {Combinatorial Structures and their Applications (Proc. Calgary Internat. Conf., 1969)},
    publisher = {Gordon and Breach},
    address = {New York},
    year = {1970},
    pages = {69--87}
}

@misc{eur2023k,
    author = {Eur, Christopher and Larson, Matt},
    title = {{$K$}-theoretic positivity for matroids},
    year = {2023},
    note = {To appear in Algebraic Geometry; arXiv:2311.11996}
}

@misc{fedorov1885nachala,
    author = {Fedorov, E. S.},
    title = {Nachala ucheniya o figurakh},
    year = {1885},
    note = {Elements of the theory of figures; reprinted, Izdat. Akad. Nauk SSSR, Moscow, 1953}
}

@article{ferroni2024examples,
    author = {Ferroni, Luis and Higashitani, Akihiro},
    title = {Examples and counterexamples in {E}hrhart theory},
    journal = {EMS Surv. Math. Sci.},
    year = {2024},
    note = {Published online first; arXiv:2307.10852},
    doi = {10.4171/EMSS/86}
}

@article{herzog2002discrete,
    author = {Herzog, J{\"u}rgen and Hibi, Takayuki},
    title = {Discrete polymatroids},
    journal = {J. Algebraic Combin.},
    volume = {16},
    number = {3},
    year = {2002},
    pages = {239--268}
}

@article{nakai1963criterion,
    author = {Nakai, Yoshikazu},
    title = {A criterion of an ample sheaf on a projective scheme},
    journal = {Amer. J. Math.},
    volume = {85},
    year = {1963},
    pages = {14--26}
}

@article{kleiman1966ampleness,
    author = {Kleiman, Steven L.},
    title = {Toward a numerical theory of ampleness},
    journal = {Ann. of Math. (2)},
    volume = {84},
    year = {1966},
    number = {3},
    pages = {293--344}
}

@misc{konoike2025magic,
    author = {Konoike, Masato},
    title = {On the magic positivity of {E}hrhart polynomials of dilated polytopes},
    year = {2025},
    note = {arXiv:2504.21395}
}

@article{lee1997formula,
    author = {Lee, Yuan-Pin},
    title = {A formula for {E}uler characteristics of tautological line bundles on the {D}eligne--{M}umford moduli spaces},
    journal = {Internat. Math. Res. Notices},
    year = {1997},
    number = {8},
    pages = {393--400}
}

@book{oxley2011matroid,
    author = {Oxley, James},
    title = {Matroid theory},
    series = {Oxford Graduate Texts in Mathematics},
    volume = {21},
    edition = {Second},
    publisher = {Oxford University Press},
    year = {2011}
}

@article{pandharipande1997symmetric,
    author = {Pandharipande, Rahul},
    title = {The symmetric function $h^0(\overline{M}_{0,n}, L_1^{x_1} \otimes \cdots \otimes L_n^{x_n})$},
    journal = {J. Algebraic Geom.},
    volume = {6},
    year = {1997},
    number = {4},
    pages = {721--731}
}

@article{postnikov2009permutohedra,
    AUTHOR = {Postnikov, Alexander},
     TITLE = {Permutohedra, associahedra, and beyond},
   JOURNAL = {Int. Math. Res. Not. IMRN},
  FJOURNAL = {International Mathematics Research Notices. IMRN},
      YEAR = {2009},
    NUMBER = {6},
     PAGES = {1026--1106},
      ISSN = {1073-7928,1687-0247},
   MRCLASS = {05E30},
  MRNUMBER = {2487491},
       DOI = {10.1093/imrn/rnn153},
       URL = {https://doi.org/10.1093/imrn/rnn153},
}

@article{reiner1997noncrossing,
    author = {Reiner, Victor},
    title = {Non-crossing partitions for classical reflection groups},
    journal = {Discrete Math.},
    volume = {177},
    year = {1997},
    pages = {195--222}
}

@article{savage2015s,
    author = {Savage, Carla D. and Visontai, Mirk\'o},
    title = {The $\mathbf{s}$-{E}ulerian polynomials have only real roots},
    journal = {Trans. Amer. Math. Soc.},
    volume = {367},
    year = {2015},
    pages = {1441--1466}
}

@article{shephard1974combinatorial,
    author = {Shephard, G. C.},
    title = {Combinatorial properties of associated zonotopes},
    journal = {Canad. J. Math.},
    volume = {26},
    year = {1974},
    pages = {302--321}
}

@article{snapper1959multiples,
    author = {Snapper, Ernst},
    title = {Multiples of divisors},
    journal = {J. Math. Mech.},
    volume = {8},
    number = {6},
    year = {1959},
    pages = {967--992}
}

@article{snapper1960polynomials,
    author = {Snapper, Ernst},
    title = {Polynomials associated with divisors},
    journal = {J. Math. Mech.},
    volume = {9},
    number = {1},
    year = {1960},
    pages = {123--139}
}

@article{simion2003type,
    author = {Simion, Rodica},
    title = {A type-{B} associahedron},
    journal = {Adv. in Appl. Math.},
    volume = {30},
    year = {2003},
    number = {1-2},
    pages = {2--25}
}

@book{szego1975orthogonal,
    author = {Szeg\H{o}, G\'abor},
    title = {Orthogonal Polynomials},
    series = {American Mathematical Society Colloquium Publications},
    volume = {23},
    edition = {Fourth},
    publisher = {American Mathematical Society, Providence, RI},
    year = {1975}
}

@incollection{stanley1991zonotope,
    author = {Stanley, Richard P.},
    title = {A zonotope associated with graphical degree sequences},
    booktitle = {Applied Geometry and Discrete Mathematics},
    series = {DIMACS Ser. Discrete Math. Theoret. Comput. Sci.},
    volume = {4},
    publisher = {Amer. Math. Soc.},
    year = {1991},
    pages = {555--570}
}

@phdthesis{steingrmsson1992permutation,
    author = {Steingr{\'\i}msson, Einar},
    title = {Permutation statistics of indexed and poset permutations},
    school = {Massachusetts Institute of Technology},
    year = {1992}
}

@article{wagner1992total,
    author = {Wagner, David G.},
    title = {Total positivity of {H}adamard products},
    journal = {J. Math. Anal. Appl.},
    volume = {163},
    year = {1992},
    pages = {459--483}
}

@book{ziegler1995lectures,
    author = {Ziegler, G{\"u}nter M.},
    title = {Lectures on Polytopes},
    series = {Graduate Texts in Mathematics},
    volume = {152},
    publisher = {Springer-Verlag, New York},
    year = {1995}
}

@incollection{branden2015unimodality,
    author = {Br{\"a}nd{\'e}n, Petter},
    title = {Unimodality, log-concavity, real-rootedness and beyond},
    booktitle = {Handbook of Enumerative Combinatorics},
    editor = {B{\'o}na, Mikl{\'o}s},
    series = {Discrete Math. Appl. (Boca Raton)},
    publisher = {CRC Press, Boca Raton, FL},
    year = {2015},
    pages = {437--483}
}

@article{silversmith2022cross,
    author = {Silversmith, Rob},
    title = {Cross-ratio degrees and perfect matchings},
    journal = {Proc. Amer. Math. Soc.},
    volume = {150},
    number = {12},
    year = {2022},
    pages = {5057--5072}
}

@article{brakensiek2025kapranov,
    author = {Brakensiek, Joshua and Eur, Christopher and Larson, Matt and Li, Shiyue},
    title = {{K}apranov degrees},
    journal = {Int. Math. Res. Not. IMRN},
    volume = {2025},
    number = {20},
    year = {2025},
    pages = {Paper No. rnaf306}
}

@article{dilworth1944dependence,
    author = {Dilworth, R. P.},
    title = {Dependence relations in a semi-modular lattice},
    journal = {Duke Math. J.},
    volume = {11},
    number = {3},
    year = {1944},
    pages = {575--587},
    doi = {10.1215/S0012-7094-44-01150-6}
}

@incollection{kung1990dilworth,
    author = {Kung, Joseph P. S.},
    title = {{D}ilworth truncations of geometric lattices},
    booktitle = {The {D}ilworth theorems: selected papers of {R}obert {P}. {D}ilworth},
    editor = {Bogart, Kenneth P. and Freese, Ralph and Kung, Joseph P. S.},
    series = {Contemporary Mathematicians},
    publisher = {Birkh\"auser},
    address = {Boston},
    year = {1990},
    pages = {295--297}
}

@book{crapo1970foundations,
    author = {Crapo, Henry H. and Rota, Gian-Carlo},
    title = {On the foundations of combinatorial theory: combinatorial geometries},
    edition = {Preliminary},
    publisher = {M.I.T. Press},
    address = {Cambridge, Mass.},
    year = {1970}
}

@incollection{mason1977matroids,
    author = {Mason, John H.},
    title = {Matroids as the study of geometrical configurations},
    booktitle = {Higher Combinatorics (Proc. NATO Advanced Study Inst., Berlin, 1976)},
    editor = {Aigner, Martin},
    publisher = {Reidel},
    address = {Dordrecht},
    year = {1977},
    pages = {133--176}
}

@book{stanley1996combinatorics,
    author = {Stanley, Richard P.},
    title = {Combinatorics and commutative algebra},
    edition = {Second},
    series = {Progress in Mathematics},
    volume = {41},
    publisher = {Birkh\"auser},
    address = {Boston},
    year = {1996}
}

@article{provan1980decompositions,
  title={Decompositions of simplicial complexes related to diameters of convex polyhedra},
  author={Provan, J Scott and Billera, Louis J},
  journal={Mathematics of Operations Research},
  volume={5},
  number={4},
  pages={576--594},
  year={1980},
  publisher={INFORMS}
}

@article{feichtner2004chow,
    author = {Feichtner, Eva Maria and Yuzvinsky, Sergey},
    title = {Chow rings of toric varieties defined by atomic lattices},
    journal = {Invent. Math.},
    volume = {155},
    number = {3},
    year = {2004},
    pages = {515--536},
    doi = {10.1007/s00222-003-0327-2}
}

@book{bjorner1999oriented,
    author = {Bj\"orner, Anders and Las Vergnas, Michel and Sturmfels, Bernd and White, Neil and Ziegler, G\"unter M.},
    title = {Oriented Matroids},
    series = {Encyclopedia of Mathematics and its Applications},
    volume = {46},
    edition = {Second},
    publisher = {Cambridge University Press, Cambridge},
    year = {1999}
}

@incollection{katz2011realization,
    author = {Katz, Eric and Payne, Sam},
    title = {Realization spaces for tropical fans},
    booktitle = {Combinatorial aspects of commutative algebra and algebraic geometry},
    series = {Abel Symp.},
    volume = {6},
    publisher = {Springer, Berlin},
    year = {2011},
    pages = {73--88},
    doi = {10.1007/978-3-642-19492-4_6}
}

@article{brion1988points,
    author = {Brion, Michel},
    title = {Points entiers dans les poly\`edres convexes},
    journal = {Ann. Sci. \'Ecole Norm. Sup. (4)},
    volume = {21},
    number = {4},
    year = {1988},
    pages = {653--663}
}

@article{li2018images,
    author = {Li, Binglin},
    title = {Images of rational maps of projective spaces},
    journal = {Int. Math. Res. Not. IMRN},
    year = {2018},
    number = {13},
    pages = {4190--4228}
}

@incollection{kapranov1993chow,
    author = {Kapranov, M. M.},
    title = {Chow quotients of {G}rassmannians. {I}},
    booktitle = {I. {M}. {G}el'fand {S}eminar},
    series = {Adv. Soviet Math.},
    volume = {16},
    publisher = {Amer. Math. Soc., Providence, RI},
    year = {1993},
    pages = {29--110}
}

@misc{eurfinklarson2025,
    author = {Eur, Christopher and Fink, Alex and Larson, Matt},
    title = {Vanishing theorems for combinatorial geometries},
    year = {2025},
    note = {arXiv:2510.05207},
    url = {https://arxiv.org/abs/2510.05207}
}

@misc{eur2026problems,
    author = {Eur, Christopher},
    title = {Problems on {$h^\ast$}-vectors for matroids},
    year = {2026},
    note = {Problem list, updated August 1, 2026},
    howpublished = {\url{https://www.math.cmu.edu/~ceur/pdf/Eur_hstar-matroid-problems.pdf}},
    url = {https://www.math.cmu.edu/~ceur/pdf/Eur_hstar-matroid-problems.pdf}
}

@misc{crowley2026cohomological,
    author = {Crowley, Colin and Larson, Matt},
    title = {Cohomological aspects of power ideals},
    year = {2026},
    note = {arXiv:2604.06601},
    url = {https://arxiv.org/abs/2604.06601}
}

@article{holtz2011zonotopal,
    author = {Holtz, Olga and Ron, Amos},
    title = {Zonotopal algebra},
    journal = {Adv. Math.},
    volume = {227},
    year = {2011},
    number = {2},
    pages = {847--894}
}

@article{ardila2010combinatorics,
    author = {Ardila, Federico and Postnikov, Alexander},
    title = {Combinatorics and geometry of power ideals},
    journal = {Trans. Amer. Math. Soc.},
    volume = {362},
    year = {2010},
    number = {8},
    pages = {4357--4384},
    doi = {10.1090/S0002-9947-10-05018-X}
}

@article{holtz2012hierarchical,
    author = {Holtz, Olga and Ron, Amos and Xu, Zhiqiang},
    title = {Hierarchical zonotopal spaces},
    journal = {Trans. Amer. Math. Soc.},
    volume = {364},
    year = {2012},
    number = {2},
    pages = {745--766}
}

@article{lenz2012hierarchical,
    author = {Lenz, Matthias},
    title = {Hierarchical zonotopal power ideals},
    journal = {European J. Combin.},
    volume = {33},
    year = {2012},
    number = {6},
    pages = {1120--1141}
}

@misc{avila2026luck,
    author = {Avila, Nicolas and Ferroni, Luis and Morales, Alejandro H.},
    title = {Luck and magic for {P}itman--{S}tanley polytopes and parking functions},
    year = {2026},
    note = {arXiv:2603.19194},
    url = {https://arxiv.org/abs/2603.19194}
}
